\documentclass[11pt]{amsart}
\usepackage[utf8]{inputenc}

\usepackage{amssymb,latexsym,amsmath,amsthm,graphicx}
\usepackage{biblatex}
\usepackage{tikz}
\usepackage{multirow}
\usepackage{float}
\usepackage{color, soul}
\restylefloat{table}
\newtheorem{theorem}{Theorem}[section]
\newtheorem{lemma}[theorem]{Lemma}
\newtheorem{proposition}[theorem]{Proposition}

\theoremstyle{definition}
\newtheorem{definition}[theorem]{Definition}
\newtheorem{example}[theorem]{Example}

\theoremstyle{remark}

\numberwithin{equation}{section}

\newcommand{\N}{\mathbb{N}}

\newcommand{\R}{\mathbb{R}}

\newcommand{\conv}{\text{conv}}
\title{Projections of Totally Disconnected thin Fractals with very thick shadows on $\R^d$}

\author{Chun-Kit Lai}
\address[Chun-Kit Lai]{Department of Mathematics, San Francisco State University,
1600 Holloway Avenue, San Francisco, CA 94132.}
 \email{cklai@sfsu.edu}

\author{Lekha Priya Patil}
\address[Lekha Priya Patil]{Department of Mathematics, San Francisco State University,
1600 Holloway Avenue, San Francisco, CA 94132.}
\email{lpatil@sfsu.edu}
\address[Lekha Priya Patil]{Department of Mathematics, University of California, Irvine,
400 Physical Sciences Quad, Irvine, CA 92697.}
 \email{patill@uci.edu}

\subjclass[2010]{28A80}
\keywords{Orthogonal Projections, Self-affine sets, Hausdorff dimensions, Rectifiable sets}

\begin{document}

\begin{abstract}
We study an extreme scenario of the Mastrand projection theorem for which a fractal has the property that its orthogonal projection is the same as the   orthogonal projection of its convex hull. We extend results in current literature and establish checkable criteria for self-affine sets to have such property. Using this, we show that every convex polytope on $\R^d$ contains a totally disconnected compact set, which is a union of self-affine sets, of dimension as close to 1 as possible, as well as a rectifiable 1-set, such that the fractal projects to an interval in every 1-dimensional subspace and its convex hull is the given polytope. Other convex sets and projections onto higher dimensional subspaces will also be discussed.      
\end{abstract}

\maketitle

\section{Introduction}
The Marstrand projection theorem \cite{marstrand1954some}, which describes the relationship between the dimensions of Borel sets and their projections through a measure-theoretic statement, is now the landmark result of the projection theory of fractals.  
\begin{theorem}[Marstrand, 1954]
    Let $E \subset \R^2$ be a Borel set, and let $\alpha = \dim_H E$. If $\alpha \leq 1$, then
    \[\dim_H(\pi_\theta(E)) = \alpha \; \; \; \; \; \; \text{   for almost every } \theta\in[0,\pi).\]
    If $\alpha > 1$, then
    \[\mathcal{L}^1(\pi_\theta((E)) > 0\; \; \; \; \; \; \text{   for almost every } \theta\in[0,\pi),\]
    where $\mathcal{L}^1$ is the Lebesgue measure. 
\end{theorem}

This result can  be generalized to higher dimensions with projections on  $k$-dimensional subspaces, see \cite{Mattila_1995} for more details. There has been a huge amount of research studying different ways of improving this theorem under various conditions. In one direction, one can obtain a more precise dimension estimate for the  exceptional set of directions. This starts from the result of Kaufman \cite{Kaufman1968} who proved that if $s<\dim_H(E)<1$, then 
$$
\dim_H\{\theta: \dim_H(\pi_{\theta}(E))<s\} \le s.
$$
There has been substantial improvement and generalization about this topic. 
In the other direction, people are interested in determining the classes of sets for which ``almost all" can be replaced by ``all" in the statement of the Mastrand's theorem.  For example,  Peres and Shmerkin \cite{PS_2009} showed that self-similar sets with dense rotation group satisfies $\dim_H(\pi_{\theta}(K))= \min \{\dim_H(K), 1\}$ for {\it all} $\theta\in [0,\pi)$. We refer the reader to \cite{falconer2015sixty} by Falconer, Fraser, and Jin for some of these recent developments. 

\medskip

This paper considers the Marstrand projection theorem in its most extreme scenario. In particular, we examine how small the Hausdorff dimension of a compact totally disconnected set on $\R^d$ can be while still ensuring its projections are exactly an interval in {\it all} one-dimensional subspaces. We aim at presenting a systematic study of such sets and give different interesting examples. Let us set up the following two terminologies:

\begin{definition} Let $K$ be a Borel set on $\R^d$ with  convex hull $C$ and let $1\le k<d$. We say that $K$ {\bf projects a very thick shadow  on $k$-dimensional subspace $W$} where $W\in G(d,k)$ if we have  $\pi_W (K) = \pi_W (C)$. We say that $K$ has the ($k$-dimensional) {\bf everywhere very thick shadow condition} if $K$ projects very thick shadows on every $W\in G(d,k)$.
\end{definition}

Here, $G(d,k)$ is the set of all $k$-dimensional subspaces in $\R^d$.  We also have a dual concept for convex sets.

\begin{definition} Let $C$ be a closed convex set in $\R^d$. We say that $C$ is \textbf{ ($k$-dimensional) fractal-decomposable} with a set $K$ if  
\begin{enumerate}
    \item $\conv (K) = C$, and
    \item $\pi_W (K) = \pi_W (C)$ for all $W \in G(d, k)$.  
\end{enumerate}
\end{definition}

Clearly, if $K$ is connected, the projection of $K$ will be connected and hence the convex hull of $K$ is 1-dimensional fractal-decomposable. Our main interest will be asking how ``disconnected" $K$ can be if it satisfies the above two definitions.  Moreover, we would hope that $K$ possesses a nice topology and geometry, so compactness of $K$ is desirable  as well as that $K$ is generated by self-similar or self-affine iterated function systems.

\medskip

  From Mastrand's theorem, we know already that the set must have Hausdorff dimension at least one. Another landmark projection theorem by Bescovitch also showed that purely unrectifiable 1-sets cannot be fractal decomposable either (See Section \ref{Section6} for the precise definitions about rectifiability). Our main result below demonstrates that other than these two constraints, all convex polytopes are fractal decomposable using compact sets of dimension close to 1 or even rectifiable 1-sets.

\begin{theorem}\label{theorem_Main1}
    Let $C$ be a convex polytope on $\R^d$. Then 
    \begin{enumerate}
        \item For all $\epsilon>0$, there exists a totally disconnected compact set $K$, with $1\le \mbox{dim}_H(K)\le 1+\epsilon$ and $K$ is a finite union of self-affine sets, such that $C$ is 1-dimensional fractal-decomposable with the set $K$. 
        \item There exists a totally disconnected compact rectifiable 1-set $K$ such that $C$ is 1-dimensional fractal-decomposable with the set $K$. 
    \end{enumerate}
\end{theorem}

The very first totally disconnected compact set which projects very thick shadows on every line was first discovered by Mendivil and Taylor \cite{mendivil2008thin}. Without noticing Mendivil-Taylor's work,  Falconer and Fraser \cite{falconer2013visible} demonstrated some examples of self-similar sets that project very thick shadows while studying the visibility conjecture of fractals. They mentioned a checking criterion for self-similar sets having very thick shadows everywhere. Farkas \cite{farkas2020interval} studied interval projections of self-similar sets and gave an example of a totally disconnected self-similar fractal of dimension arbitrarily close to 1 existing inside the unit square. In the example, rotations were required and there was no indication that the example could be generalized to other polygons or to higher dimensions. We will give an example based only on fractal squares and rigorously prove the criterion used by previous researchers on $\R^d$ for checking when a self-affine set has a very thick shadow everywhere. Moreover, as another interesting contribution of the paper, for the projection onto all 1-dimensional subspaces, we will provide another natural criterion using the connected components of the first iteration (see Theorem \ref{classification theorem1} and Theorem \ref{classification theorem2}). We can prove Theorem \ref{theorem_Main1} (1) conveniently with this new criterion. 

\medskip

The everywhere very thick shadow condition has been a useful sufficient condition to study the visibility conjecture: {\it If $F$ is Borel and $\dim_H(F)\ge 1$, then the Hausdorff dimension of the visible part of $F$ in the direction $\theta$ is equal to 1 for almost all angles $\theta$. }
Faloner and Fraser \cite{falconer2013visible} showed that the conjecture is true for self-similar sets under the convex open set condition on $\R^2$ and the everywhere very thick shadow condition. Other more general cases for the visibility conjecture of self-affine or self-similar sets were studied in \cite{Rossi,JJMSW}. Therefore, the construction given in this paper produces  examples for the visibility conjecture.

\medskip

Theorem \ref{theorem_Main1} gives a complete answer to convex polytopes. One may be interested in other convex sets as well as projections onto other subspaces of dimension greater than 1. Unfortunately, with certain simple observations, total disconnectedness is indeed an impossible requirement for fractal decomposability if we go beyond the condition for Theorem \ref{theorem_Main1}. A thorough discussion will be given in the last section of the paper.

\medskip

We will organize our paper as follows: In Section 2, we will set up our terminologies used in this paper. In Section 3, we will prove our classification theorem for self-affine sets. In Section 4, we will discuss \st{the} fractal decomposablility for unit cubes. Examples existing in current literature are also exhausted in this section. In Section 5, we  will discuss the fractal decomposablility for general polytopes using self-affine sets. In Section 6, we will discuss the fractal decomposablility for polytopes using rectifiable sets. Theorem \ref{theorem_Main1} will be proved in Section 5 and 6. Finally, we close our paper with remarks on general cases. 

,

\medskip

\section{Preliminaries}

The goal of this section is to set up the basic terminologies in this paper. Let $\Phi = \{\phi_1,...,\phi_N\}$ be a collection of contractive maps from $\R^d\to \R^d$. $\Phi$ generates an {\bf iterated function system (IFS)} with a unique non-empty compact attractor $K = K_{\Phi}.$ A {\bf self-affine} IFS consists of maps $\phi_i(x) = T_i x+t_i$ where $T_i$ are invertible matrices on $\R^d$ with the operator norm $\|T_i\|<1$ and $t_i\in\R^d$. A {\bf self-similar} IFS is a special case of a self-affine IFS in which all $T_i $ are of the form $r_i O_i$, where $0<r_i<1$ and $O_i$ is an orthogonal transformation. 

\medskip

We will use the standard multi-index notation to describe our IFS. Namely, we let $\Sigma = \{1,...,N\}$ and $\Sigma^k  = \Sigma\times...\times \Sigma$ ($k$ times). For each $\sigma =(\sigma_1,...,\sigma_k)\in \Sigma^k $, 
$$
\phi_{\sigma}(x) = \phi_{\sigma_1}\circ...\circ\phi_{\sigma_k}(x). 
$$

\medskip
 It is well-known that the Hausdorff dimension of the attractor of a self-similar IFS under the open set condition is the unique $s$ such that 
$$
\sum_{i=1}^Nr_i^s = 1.
$$
The Hausdorff dimension of a self-affine IFS is, however, much more difficult to compute. Let $\Phi = \{\phi_1, \ldots, \phi_N\}$ be a self-affine IFS. Write each $\phi_i(x) = T_ix+b_i$, where $T_i$ are linear transformations. Let $T: \R^n \to \R^n$ be a linear mapping that is contracting and nonsingular. The \textbf{singular values} $1 > \alpha_1 \geq \cdots \geq \alpha_N > 0$ of $T$ are the positive square roots of the eigenvalues of $T^*T$, where $T^*$ is the adjoint of $T$. For $0 \leq s \leq n$, we define the \textbf{singular value function} as
\[\varphi^s (T) = \alpha_1 \alpha_2 \cdots \alpha_{r-1} \alpha_r^{s-r+1},\]
where $r$ is the smallest integer greater than or equal to $s$. We define the affinity dimension  as
\begin{equation}\label{affinity}
    \dim_a(\Phi) = \inf\left\{ s: \sum_{k=1}^\infty \sum_{\sigma\in \Sigma^k} \varphi^s(T_{\sigma}) < \infty \right\} 
\end{equation}
Falconer proved that 
$$
\dim_H(K_{\Phi})\le \dim_a(\Phi)
$$
 and that equality holds for almost all translations as long as $\|T_i\|\le 1/3$ \cite{Falconer_88}. Soloymak later relaxed the condition to $\|T_i\| \le 1/2$ \cite{solo_98}. One cannot expect equality to hold generally for the well-known Bedford-McMullen carpet. In a recent deep result by B\'{a}r\'{a}ny, Hochman and Rapaport \cite{barany2019hausdorff}, however, equality does hold for affine IFS under certain natural separation and irreducibility conditions.

\medskip

\subsection{Total disconnectedness.}

    A topological space $X$ is totally disconnected if the only connected components of $X$ are the singletons. 
Many definitions of total disconnectedness appear in the literature. For the convenience of the reader, we collect all alternate equivalent definitions below that are useful for our discussions. This theorem follows from  Propositions 3.1.8 and 3.1.11 in \cite{edgar2008measure} as well as Theorem 29.7 in \cite{willard2012general} \footnote{In Theorem 29.7 of \cite{willard2012general}, rim-compact means that every point has a open set with compact boundary, which is true for a locally compact space, see p.288 of the same book}.

\begin{theorem}\label{theorem_tD} If $X$ is a locally compact separable metric space, then the following are equivalent:
\begin{enumerate}
    \item All connected components are singletons.
    \item For all $x\ne y$, there exist disjoint open sets $U, V$ such that $x\in U$ and $y\in V$ where $U \cup V = X$.
    \item For every disjoint closed sets $E, \, F$, there exist disjoint open sets $U, V$ such that $E \subset U$ and $F \subset V$ where $U \cup V = X$.
    \item There exists basis of $X$ such that every basis element is a clopen set in $X$. 
\end{enumerate}
\end{theorem}

In particular, the following proposition holds true \cite[p.91]{edgar2008measure}. 

\begin{proposition}\label{prop_totally-disconnect}
    A finite union of totally disconnected closed sets is totally disconnected. 
\end{proposition}

We note that closedness cannot be removed as the rationals and irrationals are both totally disconnected, while their union is not. We now give the following sufficient condition for total disconnectedness in fractals generated by an IFS. 

\begin{proposition} \label{0-diam-total-disconn}
    Let $\Phi = \{\phi_1, \ldots, \phi_N\}$ be an IFS. Let $C$ be a closed set such that 
    \[C \supset \bigcup_{i=1}^N \phi_i (C).\]
    For any level $k\ge 1$ we let
    \[\bigcup_{\sigma \in \Sigma^k} \phi_\sigma (C) = R_{1, k} \cup \cdots \cup R_{m_k, k}, \]
    where $R_{j, k}$ denotes the $j$th connected component of the $k$th iteration. Then, $C$ is totally disconnected if 
    \[\lim_{k \to \infty} \max_{1 \leq j \leq m_k} \mbox{diam} (R_{j, k}) = 0. \]
\end{proposition}
\begin{proof}
    Let $x, y \in K$ be distinct. Then $|x - y| > 0$. Choose the iteration $k$ where $\mbox{diam} (R_{j, k}) < |x - y|$ for all $1 \leq j \leq m_k$. Then $x$ and $y$ belong to different connected components of the $k$th level. So the connected components in the $k$th stage provides the separation. 
\end{proof}

We remark that another useful criterion for checking total disconnectedness specifically for fractal squares can be found in \cite{roinestad2010geometry} (see also \cite{LLR2013}).

\subsection{Notions about convex geometry}    We will use some notions from convex geometry. We can see \cite{gruber2007convex} for details. A hyperplane $E$ is the level set of a linear functional i.e. $E = \{x: f(x) = \alpha\}$ for some linear functional $f: \R^d\to \R$. Let $A \subset \R^d$ be closed and convex. Then $E$ is called a \textbf{supporting hyperplane} of $A$ if $A \cap E \not = \varnothing$ and $A$ is contained in exactly one of the two closed halfspaces $\{f \leq \alpha\}$ or $\{f \geq \alpha\}$.

\medskip

A {\bf convex polytope} $C$ is a closed convex set that is a convex hull of finitely many points. It is also known that a convex polytope also admits a half-space representation i.e.
$$
C = \bigcap_{i=1}^M {\mathcal H}_i
$$
where ${\mathcal H}_i$ is a closed half-space for some linear functional $f$. Recall that a point $x$ is called an \textbf{extreme point} of a convex set $C$ if $x = \lambda y+(1-\lambda)z$ for some $\lambda\in [0,1]$ and $y,z\in C$ implies that $y = z$.  A point $x$ is called an \textbf{exposed point} of a convex set $C$ if  there exists a supporting hyperplane $H$ such that $H\cap C = \{x\}$.  We would like to note that all exposed points are extreme points, but the converse is not necessarily true, see the track and field example \cite[p.75]{gruber2007convex}.

\medskip

\section{Classification theorems}

In this section we will provide a classification for which self-affine fractals in $\R^d$ project very thick shadows in every $k$-dimensional subspace, where $1 \leq k < d$. Suppose $K$ is a self-affine fractal whose IFS consists of the generating contractions
\[\{ \phi_1, \, \ldots, \, \phi_N \}.\]
Let $C = C_K$ represent the convex hull of $K$. Let $\mathcal{L}_{k} (K)$ be the set of all $d-k$ dimensional affine subspaces that intersect C. When we say {\it an affine subspace $W$} of dimension $d-k$, we mean that $W$ is a translated $d-k$ dimensional subspace, i.e.  $W = V+x$ for some $V\in G(d,d-k)$ and $x\in\R^d$.  

\medskip

The following theorem provides an equivalence by only checking the first level iteration of self-affine sets. It was first used by Mendivil and Taylor in \cite{mendivil2008thin} in their special case. It has also been mentioned for the case of $\R^2$ without proof by \cite{falconer2013visible} and \cite{farkas2020interval}. There appears to be no explicit formal proof written down. Although the proof is not too difficult to experts in fractal geometry, we provide the proof here for the sake of completeness.

\begin{theorem} \label{classification theorem1}
Using the above notations, the following are equivalent.
\begin{enumerate}
    \item For all  $W \in \mathcal{L}_k(K)$, there exists some $i \in \{1, \, \ldots, \, N\}$ such that $W \cap \phi_i(C) \not = \varnothing$.
    \item $K$ projects very thick shadows in every $k$-dimensional subspace.
    \end{enumerate}
    \end{theorem}

\begin{proof}
  \textbf{($\Longrightarrow$):} It suffices to claim that for all $W\in{\mathcal L}_k(K)$ and for all $n \in \N$, there exists some ${\sigma_n}\in \Sigma^n$ such that $\phi_{\sigma_n}(C) \cap W \not = \varnothing$. Indeed, we can take $x_n\in \phi_{\sigma_n}(K)\subset \phi_{\sigma_n}(C)$. By passing into subsequence if necessary, we can assume $x_n\to x\in K$ and we have
    $$
    \text{dist}(x_{n}, \, W) \leq  \text{diam}(\phi_{\sigma_n}(C))\to0
    $$
    as $n\to \infty$. Hence, $x\in W.$  As all $d-k$ dimensional affine subspaces hits $K$, the projection of $K$ onto any $k$ dimensional subspace must be equal to that of $C$. This will prove (2). 

\medskip

    We now justify our claim by induction. Fix $W \in \mathcal{L}_k(K)$, By our assumption (1), $\ell \cap \phi_i(C) \not = \varnothing$ for some $i \in \{1, \, 2, \, \ldots, \, N\}$. Taking $\sigma_1 = i$, we proved the case $n=1$. 
    For an inductive hypothesis, assume that at the $n$th iteration, there exists some map $\phi_{\sigma_n}$ such that $\ell \cap \phi_{\sigma_n}(C) \not = \varnothing$. 
    Because the IFS is self-affine, the $n$th level map, $\phi_{\sigma_n}$, from the hypothesis must have an inverse, $\phi_{\sigma_n}^{-1}$.    From our inductive hypothesis,  we have $\phi_{\sigma_n}(C) \cap W \not = \varnothing.$ It  implies that  $C \cap \phi_{\sigma_n}^{-1}(W) \not = \varnothing$.
    Notice that  $\phi_{\sigma_n}^{-1}(W)$ is still a $(d-k)$ dimensional affine subspace, so the above implies that $ \phi_{\sigma_n}^{-1}(W)\in \mathcal{L}_k(K)$. So, by our assumption, there exists some $i \in \{1, \, \ldots, \, N\}$ such that 
        \[ \phi_i(C) \cap \phi_{\sigma_n}^{-1}(W) \not = \varnothing.\]
    Applying the forward map $\phi_{\sigma_n}$ to the above, we have
    \begin{align*}
        \phi_{\sigma_n}(\phi_i(C)) \cap W &\not =  \varnothing.
    \end{align*}
Hence, for $\phi_{\sigma_{n+1}} = \phi_{\sigma_n} \circ \phi_i$, we have  $ \phi_{\sigma_{n+1}}(C)\cap W\ne \varnothing$. So we have shown that $(1) \implies (2)$ holds.

\medskip

    \textbf{($\Longleftarrow$): } 
    Suppose that the self-affine fractal $K$ has a very thick shadow in every $k$-dimensional subspace. Then we know that $\pi_V(K) = \pi_V (C)$ for all $V\in G(d,k)$.
    Fix $W \in \mathcal{L}_k(K)$. Take $V$ to be the $k$-dimensional subspace orthogonal to $W$. As $W$ passes through $C$, we can find $z\in W\cap C$. Let  $y$ be the orthogonal projection of $z$ onto $V$, then we have that $y\in V\cap \pi_V(C)$. Moreover, $\pi_V^{-1}(\{y\}) = W \cap C$.  By our assumption that $\pi_V(C) = \pi_V(K)$, there exists $x\in K$ such that $\pi_V(x) = y$. Hence, $x\in W$. Recall that 
        \[K \subset \bigcup_{i=1}^N \phi_i(C). \]
    Therefore, for some $i \in \{1, \, \ldots, \, N\}$, we have $x \in \phi_i(C)$. In particular, this shows that $\phi_i(C)\cap W\neq \varnothing.$  
\end{proof}

\medskip

 Checking condition (2) in Theorem \ref{classification theorem1} may still be a tough computational task. The following theorem gives a third equivalent condition when $k = 1$ using the convex hulls of subsets of the connected components of the first iteration of the generating IFS. This third condition is a new contribution, to the best of our knowledge. Moreover, as we will see, this will provide a much more effective checking criterion for self-affine fractals. 

\medskip    

To state the theorem, we now decompose the image of the first iteration of the convex hull into the disjoint union of $r$ connected components, denoted by $R_j$ where $j \in \{1, \, \ldots, \, r\}$. That is,    
\[\bigcup_{i=1}^N \phi_i(C) = R_1 \cup \cdots \cup R_r.   \]
We are now ready to state our main theorem. 

\begin{theorem} \label{classification theorem2} Using the  prior notations, the following are equivalent:
\begin{enumerate}

    \item $K$ has a very thick shadow in every 1-dimensional subspace. 

    \item For all proper indexing sets $I \subset \{1, \, \ldots, \, r\}$,
    \[\conv\left( \bigcup_{i \in I} R_i \right) \cap \; \conv\left( \bigcup_{i \in I^c} R_i \right) \not = \varnothing. \]
    
\end{enumerate}

\end{theorem}


\begin{proof}
  
    \textbf{($\Longrightarrow$):} 
    Assume $K$ has a very thick shadow  in every 1-dimensional subspace. Then for all $W\in G(2,1)$, we have $\pi_W(K) = \pi_W (C)$is an interval, where $\conv(K) = C$. Suppose for a contradiction, there exists some proper indexing set $I \subset \{1, \, \ldots, \, r\}$ such that 
        \[\conv \left( \bigcup_{i \in I} R_i \right) \cap \conv \left( \bigcup_{i \in I^c} R_i \right) = \varnothing.\]
    For ease, we will let $A$ and $B$ denote the above sets respectively. 
    By the hyperplane separation theorem, the convex nonempty sets $A$ and $B$ can be properly separated by some  hyperplane  $H = \{x: \langle n,x\rangle = c\}$. This implies that $A\subset H^- = \{x: \langle n,x\rangle < c\}$ and $B\subset H^+ = \{x: \langle n,x\rangle > c\}$. Let $W$ be the 1-dimensional subspace spanned by $n$. If we project $A$ and $B$ onto $W$, then $\pi_W(A)\cap \pi_W(B) = \varnothing$. Thus, $\pi_W(A)\cup\pi_W(B)$ is a union of two disjoint closed intervals. But $K$ is a subset of $\bigcup_{i=1}^r R_i$, 
    implying that 
    $$
    \pi_W(K)\subset \pi_W(A)\cup \pi_W(B).
    $$ 
    As $I$ is proper, $K$ intersects both $A$ and $B$. 
    Hence, $\pi_W(K)$ cannot be an interval, contradicting our assumption (1). 
       

\medskip

    \textbf{($\Longleftarrow$)}
    Assume that for any indexing set $I \subset \{1, \, 2, \, \ldots, \, r\}$, we have that 
        \[\conv \left( \bigcup_{i \in I} R_i \right) \cap \; \conv \left( \bigcup_{i \in I^c} R_i \right) \not = \varnothing.\]
    For a contradiction, suppose that the condition of very thick shadows fails. Bythe contrapositive of Theorem \ref{classification theorem1}, there exists some hyperplane $H \in \mathcal{L}_1(K)$ such that for all $i \in \{1, \, \ldots, \, N\}$, we have $H \cap \phi_i(C) = \varnothing.$
    Note that each $\phi_i(C)$ is connected. The definition of connected components implies that 
        \[R_j = \bigcup_{\{i: \phi_i(C) \subset R_j\}} \phi_i(C). \]
    We see that $H$ cannot intersect any connected component $R_1, \ldots, R_r$.
    So, there must exist some indexing set $I \subset \{1, \, \ldots, \, r\}$ such that
        \[\bigcup_{i \in I} R_i \subset H^+ \; \text{ and } \; \bigcup_{i \in I^c} R_i \subset H^- \]
        where $H^+,H^-$ are the upper and lower half plane determined by $H$, Moreover, $I\neq \varnothing$ because  if all $R_i\in H^{+}$ or $H^-$, then $H$ will not intersect $C$. 
    Since the convex hull of a set is the smallest convex set containing the original set and $H^{\pm}$ are convex, we see that 
        \[ \conv \left( \bigcup_{i \in I} R_i \right) \subset H^+ \
    \mbox{and}
        \  \subset \conv \left( \bigcup_{i \in I^c} R_i \right) \subset H^-.\]
    Therefore, there does exist some proper indexing set $I$ such that 
        \[\conv \left( \bigcup_{i \in I} R_i \right) \cap \conv \left( \bigcup_{i \in I^c} R_i \right) = \varnothing.\]
    This contradicts our original assumption. The proof is complete. 
\end{proof}

\medskip

\section{Fractal Decomposability for cubes}

Let us first revisit some known examples in literature and provide a full characterization, for those fractals, when they have very thick shadows in all 1-dimensional subspaces. All these examples were on the unit square.  In the last example, we provide a new construction which works in higher dimension that shows the fractal decomposability for the unit cubes in $\R^d$.

\medskip

\underline{\it 1. Mendivil-Taylor self-affine fractals} 

\medskip

Mendivil and Taylor first discovered a self-affine fractal whose convex hull is the unit square such that it has very thick shadows in every projection. Let $0<t<\frac12<s<1$ and $s+t<1$, the contractive maps are defined by 
$$
\phi_1(x,y) = (tx,sy), \ \phi_2(x,y) = (sx, ty+(1-t)), 
$$
 $$
 \phi_3(x,y) = (sx+(1-s),ty), \ \phi_4(x,y) = (tx+(1-x), sy+(1-s))
 $$
 They obtained a sufficient condition on $t,s$ such that the corresponding invariant set projects very thick shadows in every direction. We now let $R_i = \phi_i([0,1]^2)$, that is, $R_i$ is the rectangle that is the image of $\phi_i$ applied on the unit square. In the following argument, we prove that the condition given by Mendivil and Taylor is a complete description of $t,s$ for the everywhere very thick shadow condition. 

\begin{example}
Let $0<t<\frac12<s<1$ and $s+t<1$. Then the Mendivil-Taylor self-affine set has very thick shadows in all directions if and only if 
$$
 \frac{1-\sqrt{2s-1}}{2}\le t.
$$

\end{example}
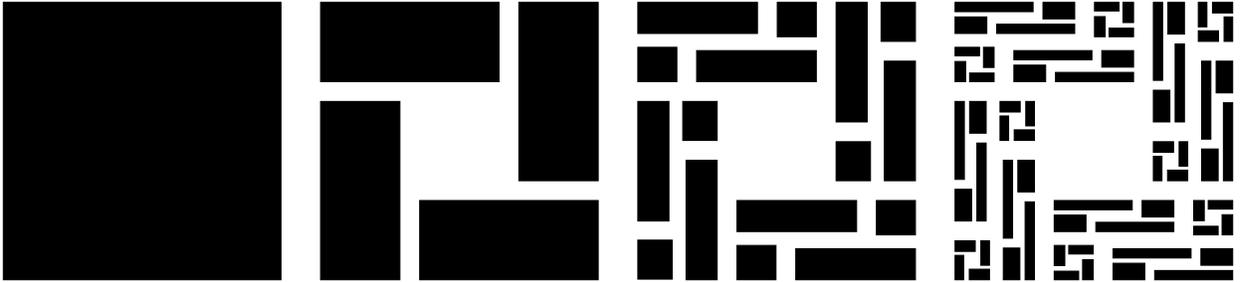
\begin{figure}[h]
\begin{tikzpicture}[x=0.75pt,y=0.75pt,yscale=-1,xscale=1]

\draw  [fill={rgb, 255:red, 0; green, 0; blue, 0 }  ,fill opacity=1 ] (10,50) -- (150,50) -- (150,190) -- (10,190) -- cycle ;
\draw  [fill={rgb, 255:red, 0; green, 0; blue, 0 }  ,fill opacity=1 ] (310,50) -- (310,140) -- (270,140) -- (270,50) -- cycle ;
\draw  [fill={rgb, 255:red, 0; green, 0; blue, 0 }  ,fill opacity=1 ] (210,100) -- (210,190) -- (170,190) -- (170,100) -- cycle ;
\draw  [fill={rgb, 255:red, 0; green, 0; blue, 0 }  ,fill opacity=1 ] (220,150) -- (310,150) -- (310,190) -- (220,190) -- cycle ;
\draw  [fill={rgb, 255:red, 0; green, 0; blue, 0 }  ,fill opacity=1 ] (170,50) -- (260,50) -- (260,90) -- (170,90) -- cycle ;
\draw  [fill={rgb, 255:red, 0; green, 0; blue, 0 }  ,fill opacity=1 ] (330,50) -- (390.33,50) -- (390.33,65.67) -- (330,65.67) -- cycle ;
\draw  [fill={rgb, 255:red, 0; green, 0; blue, 0 }  ,fill opacity=1 ] (359.67,74.33) -- (420,74.33) -- (420,90) -- (359.67,90) -- cycle ;
\draw  [fill={rgb, 255:red, 0; green, 0; blue, 0 }  ,fill opacity=1 ] (409.67,174.33) -- (470,174.33) -- (470,190) -- (409.67,190) -- cycle ;
\draw  [fill={rgb, 255:red, 0; green, 0; blue, 0 }  ,fill opacity=1 ] (380,150) -- (440.33,150) -- (440.33,165.67) -- (380,165.67) -- cycle ;
\draw  [fill={rgb, 255:red, 0; green, 0; blue, 0 }  ,fill opacity=1 ] (445.67,50) -- (445.67,110.33) -- (430,110.33) -- (430,50) -- cycle ;
\draw  [fill={rgb, 255:red, 0; green, 0; blue, 0 }  ,fill opacity=1 ] (470,79.67) -- (470,140) -- (454.33,140) -- (454.33,79.67) -- cycle ;
\draw  [fill={rgb, 255:red, 0; green, 0; blue, 0 }  ,fill opacity=1 ] (370,129.67) -- (370,190) -- (354.33,190) -- (354.33,129.67) -- cycle ;
\draw  [fill={rgb, 255:red, 0; green, 0; blue, 0 }  ,fill opacity=1 ] (345.67,100) -- (345.67,160.33) -- (330,160.33) -- (330,100) -- cycle ;
\draw  [fill={rgb, 255:red, 0; green, 0; blue, 0 }  ,fill opacity=1 ] (400.33,50) -- (420,50) -- (420,67.33) -- (400.33,67.33) -- cycle ;
\draw  [fill={rgb, 255:red, 0; green, 0; blue, 0 }  ,fill opacity=1 ] (380,172.67) -- (399.67,172.67) -- (399.67,190) -- (380,190) -- cycle ;
\draw  [fill={rgb, 255:red, 0; green, 0; blue, 0 }  ,fill opacity=1 ] (450.33,150) -- (470,150) -- (470,167.33) -- (450.33,167.33) -- cycle ;
\draw  [fill={rgb, 255:red, 0; green, 0; blue, 0 }  ,fill opacity=1 ] (330,72.67) -- (349.67,72.67) -- (349.67,90) -- (330,90) -- cycle ;
\draw  [fill={rgb, 255:red, 0; green, 0; blue, 0 }  ,fill opacity=1 ] (470,50) -- (470,69.67) -- (452.67,69.67) -- (452.67,50) -- cycle ;
\draw  [fill={rgb, 255:red, 0; green, 0; blue, 0 }  ,fill opacity=1 ] (347.33,170) -- (347.33,189.67) -- (330,189.67) -- (330,170) -- cycle ;
\draw  [fill={rgb, 255:red, 0; green, 0; blue, 0 }  ,fill opacity=1 ] (370,100) -- (370,119.67) -- (352.67,119.67) -- (352.67,100) -- cycle ;
\draw  [fill={rgb, 255:red, 0; green, 0; blue, 0 }  ,fill opacity=1 ] (447.33,120.33) -- (447.33,140) -- (430,140) -- (430,120.33) -- cycle ;
\draw  [fill={rgb, 255:red, 0; green, 0; blue, 0 }  ,fill opacity=1 ] (490,50) -- (529.33,50) -- (529.33,54.67) -- (490,54.67) -- cycle ;
\draw  [fill={rgb, 255:red, 0; green, 0; blue, 0 }  ,fill opacity=1 ] (511,61) -- (550.33,61) -- (550.33,65.67) -- (511,65.67) -- cycle ;
\draw  [fill={rgb, 255:red, 0; green, 0; blue, 0 }  ,fill opacity=1 ] (519.67,74.33) -- (559,74.33) -- (559,79) -- (519.67,79) -- cycle ;
\draw  [fill={rgb, 255:red, 0; green, 0; blue, 0 }  ,fill opacity=1 ] (590.67,185.33) -- (630,185.33) -- (630,190) -- (590.67,190) -- cycle ;
\draw  [fill={rgb, 255:red, 0; green, 0; blue, 0 }  ,fill opacity=1 ] (569.67,174.33) -- (609,174.33) -- (609,179) -- (569.67,179) -- cycle ;
\draw  [fill={rgb, 255:red, 0; green, 0; blue, 0 }  ,fill opacity=1 ] (540,150) -- (579.33,150) -- (579.33,154.67) -- (540,154.67) -- cycle ;
\draw  [fill={rgb, 255:red, 0; green, 0; blue, 0 }  ,fill opacity=1 ] (561,161) -- (600.33,161) -- (600.33,165.67) -- (561,165.67) -- cycle ;
\draw  [fill={rgb, 255:red, 0; green, 0; blue, 0 }  ,fill opacity=1 ] (540.67,85.33) -- (580,85.33) -- (580,90) -- (540.67,90) -- cycle ;
\draw  [fill={rgb, 255:red, 0; green, 0; blue, 0 }  ,fill opacity=1 ] (494.67,100) -- (494.67,139.33) -- (490,139.33) -- (490,100) -- cycle ;
\draw  [fill={rgb, 255:red, 0; green, 0; blue, 0 }  ,fill opacity=1 ] (519,129.67) -- (519,169) -- (514.33,169) -- (514.33,129.67) -- cycle ;
\draw  [fill={rgb, 255:red, 0; green, 0; blue, 0 }  ,fill opacity=1 ] (505.67,121) -- (505.67,160.33) -- (501,160.33) -- (501,121) -- cycle ;
\draw  [fill={rgb, 255:red, 0; green, 0; blue, 0 }  ,fill opacity=1 ] (630,100.67) -- (630,140) -- (625.33,140) -- (625.33,100.67) -- cycle ;
\draw  [fill={rgb, 255:red, 0; green, 0; blue, 0 }  ,fill opacity=1 ] (619,79.67) -- (619,119) -- (614.33,119) -- (614.33,79.67) -- cycle ;
\draw  [fill={rgb, 255:red, 0; green, 0; blue, 0 }  ,fill opacity=1 ] (605.67,71) -- (605.67,110.33) -- (601,110.33) -- (601,71) -- cycle ;
\draw  [fill={rgb, 255:red, 0; green, 0; blue, 0 }  ,fill opacity=1 ] (530,150.67) -- (530,190) -- (525.33,190) -- (525.33,150.67) -- cycle ;
\draw  [fill={rgb, 255:red, 0; green, 0; blue, 0 }  ,fill opacity=1 ] (594.67,50) -- (594.67,89.33) -- (590,89.33) -- (590,50) -- cycle ;
\draw  [fill={rgb, 255:red, 0; green, 0; blue, 0 }  ,fill opacity=1 ] (534.33,50) -- (550.33,50) -- (550.33,58.33) -- (534.33,58.33) -- cycle ;
\draw  [fill={rgb, 255:red, 0; green, 0; blue, 0 }  ,fill opacity=1 ] (569.67,181.67) -- (585.67,181.67) -- (585.67,190) -- (569.67,190) -- cycle ;
\draw  [fill={rgb, 255:red, 0; green, 0; blue, 0 }  ,fill opacity=1 ] (614,174.33) -- (630,174.33) -- (630,182.67) -- (614,182.67) -- cycle ;
\draw  [fill={rgb, 255:red, 0; green, 0; blue, 0 }  ,fill opacity=1 ] (584.33,150) -- (600.33,150) -- (600.33,158.33) -- (584.33,158.33) -- cycle ;
\draw  [fill={rgb, 255:red, 0; green, 0; blue, 0 }  ,fill opacity=1 ] (564,74.33) -- (580,74.33) -- (580,82.67) -- (564,82.67) -- cycle ;
\draw  [fill={rgb, 255:red, 0; green, 0; blue, 0 }  ,fill opacity=1 ] (540,157.33) -- (556,157.33) -- (556,165.67) -- (540,165.67) -- cycle ;
\draw  [fill={rgb, 255:red, 0; green, 0; blue, 0 }  ,fill opacity=1 ] (519.67,81.67) -- (535.67,81.67) -- (535.67,90) -- (519.67,90) -- cycle ;
\draw  [fill={rgb, 255:red, 0; green, 0; blue, 0 }  ,fill opacity=1 ] (490,57.33) -- (506,57.33) -- (506,65.67) -- (490,65.67) -- cycle ;
\draw  [fill={rgb, 255:red, 0; green, 0; blue, 0 }  ,fill opacity=1 ] (598.33,94.33) -- (598.33,110.33) -- (590,110.33) -- (590,94.33) -- cycle ;
\draw  [fill={rgb, 255:red, 0; green, 0; blue, 0 }  ,fill opacity=1 ] (605.67,50) -- (605.67,66) -- (597.33,66) -- (597.33,50) -- cycle ;
\draw  [fill={rgb, 255:red, 0; green, 0; blue, 0 }  ,fill opacity=1 ] (630,79.67) -- (630,95.67) -- (621.67,95.67) -- (621.67,79.67) -- cycle ;
\draw  [fill={rgb, 255:red, 0; green, 0; blue, 0 }  ,fill opacity=1 ] (622.67,124) -- (622.67,140) -- (614.33,140) -- (614.33,124) -- cycle ;
\draw  [fill={rgb, 255:red, 0; green, 0; blue, 0 }  ,fill opacity=1 ] (522.67,174) -- (522.67,190) -- (514.33,190) -- (514.33,174) -- cycle ;
\draw  [fill={rgb, 255:red, 0; green, 0; blue, 0 }  ,fill opacity=1 ] (530,129.67) -- (530,145.67) -- (521.67,145.67) -- (521.67,129.67) -- cycle ;
\draw  [fill={rgb, 255:red, 0; green, 0; blue, 0 }  ,fill opacity=1 ] (498.33,144.33) -- (498.33,160.33) -- (490,160.33) -- (490,144.33) -- cycle ;
\draw  [fill={rgb, 255:red, 0; green, 0; blue, 0 }  ,fill opacity=1 ] (505.67,100) -- (505.67,116) -- (497.33,116) -- (497.33,100) -- cycle ;
\draw  [fill={rgb, 255:red, 0; green, 0; blue, 0 }  ,fill opacity=1 ] (502.33,72.67) -- (502.33,77) -- (490,77) -- (490,72.67) -- cycle ;
\draw  [fill={rgb, 255:red, 0; green, 0; blue, 0 }  ,fill opacity=1 ] (572.67,50) -- (572.67,54.33) -- (560.33,54.33) -- (560.33,50) -- cycle ;
\draw  [fill={rgb, 255:red, 0; green, 0; blue, 0 }  ,fill opacity=1 ] (580,63) -- (580,67.33) -- (567.67,67.33) -- (567.67,63) -- cycle ;
\draw  [fill={rgb, 255:red, 0; green, 0; blue, 0 }  ,fill opacity=1 ] (509.67,85.67) -- (509.67,90) -- (497.33,90) -- (497.33,85.67) -- cycle ;
\draw  [fill={rgb, 255:red, 0; green, 0; blue, 0 }  ,fill opacity=1 ] (552.33,185.67) -- (552.33,190) -- (540,190) -- (540,185.67) -- cycle ;
\draw  [fill={rgb, 255:red, 0; green, 0; blue, 0 }  ,fill opacity=1 ] (559.67,172.67) -- (559.67,177) -- (547.33,177) -- (547.33,172.67) -- cycle ;
\draw  [fill={rgb, 255:red, 0; green, 0; blue, 0 }  ,fill opacity=1 ] (622.67,163) -- (622.67,167.33) -- (610.33,167.33) -- (610.33,163) -- cycle ;
\draw  [fill={rgb, 255:red, 0; green, 0; blue, 0 }  ,fill opacity=1 ] (630,150) -- (630,154.33) -- (617.67,154.33) -- (617.67,150) -- cycle ;
\draw  [fill={rgb, 255:red, 0; green, 0; blue, 0 }  ,fill opacity=1 ] (612.67,50) -- (617,50) -- (617,62.33) -- (612.67,62.33) -- cycle ;
\draw  [fill={rgb, 255:red, 0; green, 0; blue, 0 }  ,fill opacity=1 ] (490,177.67) -- (494.33,177.67) -- (494.33,190) -- (490,190) -- cycle ;
\draw  [fill={rgb, 255:red, 0; green, 0; blue, 0 }  ,fill opacity=1 ] (503,170.33) -- (507.33,170.33) -- (507.33,182.67) -- (503,182.67) -- cycle ;
\draw  [fill={rgb, 255:red, 0; green, 0; blue, 0 }  ,fill opacity=1 ] (525.67,100) -- (530,100) -- (530,112.33) -- (525.67,112.33) -- cycle ;
\draw  [fill={rgb, 255:red, 0; green, 0; blue, 0 }  ,fill opacity=1 ] (512.67,107.33) -- (517,107.33) -- (517,119.67) -- (512.67,119.67) -- cycle ;
\draw  [fill={rgb, 255:red, 0; green, 0; blue, 0 }  ,fill opacity=1 ] (603,120.33) -- (607.33,120.33) -- (607.33,132.67) -- (603,132.67) -- cycle ;
\draw  [fill={rgb, 255:red, 0; green, 0; blue, 0 }  ,fill opacity=1 ] (590,127.67) -- (594.33,127.67) -- (594.33,140) -- (590,140) -- cycle ;
\draw  [fill={rgb, 255:red, 0; green, 0; blue, 0 }  ,fill opacity=1 ] (625.67,57.33) -- (630,57.33) -- (630,69.67) -- (625.67,69.67) -- cycle ;
\draw  [fill={rgb, 255:red, 0; green, 0; blue, 0 }  ,fill opacity=1 ] (574.67,50) -- (580,50) -- (580,60.17) -- (574.67,60.17) -- cycle ;
\draw  [fill={rgb, 255:red, 0; green, 0; blue, 0 }  ,fill opacity=1 ] (554.33,179.83) -- (559.67,179.83) -- (559.67,190) -- (554.33,190) -- cycle ;
\draw  [fill={rgb, 255:red, 0; green, 0; blue, 0 }  ,fill opacity=1 ] (540,172.67) -- (545.33,172.67) -- (545.33,182.83) -- (540,182.83) -- cycle ;
\draw  [fill={rgb, 255:red, 0; green, 0; blue, 0 }  ,fill opacity=1 ] (624.67,157.17) -- (630,157.17) -- (630,167.33) -- (624.67,167.33) -- cycle ;
\draw  [fill={rgb, 255:red, 0; green, 0; blue, 0 }  ,fill opacity=1 ] (610.33,150) -- (615.67,150) -- (615.67,160.17) -- (610.33,160.17) -- cycle ;
\draw  [fill={rgb, 255:red, 0; green, 0; blue, 0 }  ,fill opacity=1 ] (490,79.83) -- (495.33,79.83) -- (495.33,90) -- (490,90) -- cycle ;
\draw  [fill={rgb, 255:red, 0; green, 0; blue, 0 }  ,fill opacity=1 ] (504.33,72.67) -- (509.67,72.67) -- (509.67,82.83) -- (504.33,82.83) -- cycle ;
\draw  [fill={rgb, 255:red, 0; green, 0; blue, 0 }  ,fill opacity=1 ] (560.33,57.17) -- (565.67,57.17) -- (565.67,67.33) -- (560.33,67.33) -- cycle ;
\draw  [fill={rgb, 255:red, 0; green, 0; blue, 0 }  ,fill opacity=1 ] (630,50) -- (630,55.33) -- (619.83,55.33) -- (619.83,50) -- cycle ;
\draw  [fill={rgb, 255:red, 0; green, 0; blue, 0 }  ,fill opacity=1 ] (507.33,184.67) -- (507.33,190) -- (497.17,190) -- (497.17,184.67) -- cycle ;
\draw  [fill={rgb, 255:red, 0; green, 0; blue, 0 }  ,fill opacity=1 ] (500.17,170.33) -- (500.17,175.67) -- (490,175.67) -- (490,170.33) -- cycle ;
\draw  [fill={rgb, 255:red, 0; green, 0; blue, 0 }  ,fill opacity=1 ] (530,114.33) -- (530,119.67) -- (519.83,119.67) -- (519.83,114.33) -- cycle ;
\draw  [fill={rgb, 255:red, 0; green, 0; blue, 0 }  ,fill opacity=1 ] (522.83,100) -- (522.83,105.33) -- (512.67,105.33) -- (512.67,100) -- cycle ;
\draw  [fill={rgb, 255:red, 0; green, 0; blue, 0 }  ,fill opacity=1 ] (607.33,134.67) -- (607.33,140) -- (597.17,140) -- (597.17,134.67) -- cycle ;
\draw  [fill={rgb, 255:red, 0; green, 0; blue, 0 }  ,fill opacity=1 ] (600.17,120.33) -- (600.17,125.67) -- (590,125.67) -- (590,120.33) -- cycle ;
\draw  [fill={rgb, 255:red, 0; green, 0; blue, 0 }  ,fill opacity=1 ] (622.83,64.33) -- (622.83,69.67) -- (612.67,69.67) -- (612.67,64.33) -- cycle ;

\end{tikzpicture}
    \caption{Mendivil-Taylor fractal}
    \label{fig:Mendivil-Taylor-fractal}
\end{figure}

 \begin{proof}
 We use the characterization about the convex hull in Theorem \ref{classification theorem2}. By the rotational symmetry of the rectangles inside the square, Condition (3) holds if and only if 
 \begin{enumerate}
 \item[(i)] The rectangle $R_1$ intersects the convex hull of the other three rectangles.
 \item[(ii)] The convex hull of $R_1\cup R_3$ intersects the convex hull of $R_2\cup R_4$. 
 \end{enumerate}
Denote the straight line joining $(1-s,0)$ and $(0,1-t)$  by
$$
\ell_1: f(x,y) = (1-t)x+ (1-s)y = (1-t)(1-s)
$$
Similarly, denote the straight line joining $(t,0)$ and $(s,1-t)$ by 
$$
\ell_2: g(x,y) = (1-t)x-(s-t)y = (1-t)t
$$
(i) is achieved if and only if the right upper corner point $(t,s)$ of $R_1$ lies above $\ell_1$, meaning that $f(t,s)\ge (1-t)(1-s)$. Plugging in and rearranging, we have 
$$
(1-t)(1-s-t)\le s(1-s).
$$
Similarly, (ii) is achieved if and only if $g(1-s,t) \le t(1-t)$, which means that 
$$
(1-t)(1-s-t)\le t(s-t).
$$
Indeed, for $s+t\le 1$, we have $t(s-t)\le s(1-s)$, so (ii) implies (i). Thus, the Mendivil-Taylor fractal has a very thick shadow in every direction if and only if $(1-t)(1-s-t)\le t(s-t)$. Solving for $t$, we obtain the desired characterization.

 \end{proof}

\medskip

\underline{\it 2. A Rotated square in the middle} 

\medskip

This example was considered by Falconer and Fraser in \cite{falconer2013visible} and Farkas in \cite{farkas2020interval}. 

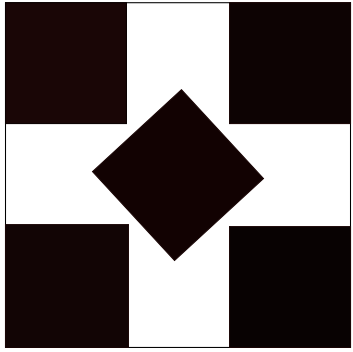
\begin{figure}[h]
 
\begin{tikzpicture}[x=0.75pt,y=0.75pt,yscale=-1,xscale=1]

\draw   (222,67) -- (396,67) -- (396,241) -- (222,241) -- cycle ;
\draw  [color={rgb, 255:red, 11; green, 1; blue, 1 }  ,draw opacity=1 ][fill={rgb, 255:red, 26; green, 6; blue, 6 }  ,fill opacity=1 ] (222,67) -- (283,67) -- (283,128) -- (222,128) -- cycle ;
\draw  [color={rgb, 255:red, 26; green, 6; blue, 6 }  ,draw opacity=1 ][fill={rgb, 255:red, 8; green, 2; blue, 2 }  ,fill opacity=1 ] (335,180) -- (396,180) -- (396,241) -- (335,241) -- cycle ;
\draw  [color={rgb, 255:red, 26; green, 6; blue, 6 }  ,draw opacity=1 ][fill={rgb, 255:red, 18; green, 5; blue, 5 }  ,fill opacity=1 ] (222,179) -- (284,179) -- (284,241) -- (222,241) -- cycle ;
\draw  [color={rgb, 255:red, 26; green, 6; blue, 6 }  ,draw opacity=1 ][fill={rgb, 255:red, 13; green, 3; blue, 3 }  ,fill opacity=1 ] (335,67) -- (396,67) -- (396,128) -- (335,128) -- cycle ;
\draw  [color={rgb, 255:red, 26; green, 6; blue, 6 }  ,draw opacity=1 ][fill={rgb, 255:red, 18; green, 2; blue, 2 }  ,fill opacity=1 ] (265.9,152.25) -- (310.75,110.9) -- (352.1,155.75) -- (307.25,197.1) -- cycle ;

\end{tikzpicture}

  \caption{IFS for Example \ref{ex4.2}}
    \label{fig:Mendivil-Taylor-fractal}
\end{figure}

\begin{example}\label{ex4.2}
Let $0<r<1/2$. Consider the self-similar IFS on the unit square generated by 
$$
\phi_1(x) = rx, \ \phi_2(x) = rx+ (1-r,0), \ \phi_3(x) = rx+(0,1-r), \ \phi_4(x) = rx+ (1-r,1-r), 
$$
$$
\phi_5(x) = rR_{\pi/4}x+ \left(\frac{1}{2}, \frac12-\frac{\sqrt{2}}{2}r\right)
$$
Then (1) the attractor projects very thick shadows in every 1-dimensional subspace if and only if $r\ge 1/3$, and (2) the attractor is totally disconnected if $r<\frac{1}{2+\frac{1}{\sqrt{2}}}$. 
\end{example}

 \begin{proof}
 (1). We will apply (2) from Theorem \ref{classification theorem2}. Notice that the attractor projects very thick shadows if and only if the corner square always intersects the convex hull of the remaining four squares. A direct check shows that the rotated square is always inside the convex hull of the three corner squares. Hence, by symmetry, the condition is equivalent to the right hand corner of the square $[0,r]^2$, i.e. $(r,r)$ lying above the line $x+y = 1-r$, which is the line joining $(0,1-r)$ and $(1-r,0)$. Hence, $r+r\ge 1-r$, meaning $r\ge 1/3.$

\medskip
(2). By ensuring that the corner $(r,r)$ not intersecting the rotated square, which is exactly when $r<\frac{1}{2+\frac{1}{\sqrt{2}}}$,  all squares will not intersect each other.  Hence, the fractal will be totally disconnected.  
 
 \end{proof}

Consequently, a totally disconnected self-similar set projecting very thick shadows in every 1-dimensional subspace exists if we take $1/3\le r < \frac{1}{2+1/{\sqrt{2}}}$. Farkas indicated how to use this pattern to create self-similar totally disconnected sets whose Hausdorff dimension is arbitrarily close to 1. However, there were no explicit mappings written down and it is unclear how one can generalize into higher dimensions. In the following example, we provide an explicit solution  based on fractal squares/cubes. 

\medskip

\underline{\it 3. Cross and Corner Fractal cubes} 

\begin{example}\label{cross_corner}
    
[Cross and Corner Fractal Cubes]
     We introduce a fractal cube construction that we will call the \textit{cross and corner fractal}. To begin, fix some odd positive integer $n$. Divide the unit cube into $n \times \cdots\times n$ smaller cubes on $\R^d$. We will denote by ${\mathcal P}_d$ the set of all $d\times d$ permutation matrices that map the coordinate hyperplane $x_d = 0$ to another coordinate plane $x_j = 0$ for some $j = 1,...,d$. 

\medskip

 Take the cross in the center of our grid, whose digits are
     $$
     { {\mathtt Cross}} = \bigcup_{\sigma\in{\mathcal P}_d}\sigma\left\{(j,\frac{n-1}{2},..., \frac{n-1}{2}) : j = 1, \ldots , n-2\right\}.
     $$

We will say the {\it diagonals} are all the unit cubes, such that each of them is connected only by a corner, from $(\epsilon_1,...,\epsilon_d)$ to $(1-\epsilon_1,...,1-\epsilon_d)$, for $\epsilon_1,...,\epsilon_d\in\{0,1\}$.    

\medskip

For each corner of the unit cube,  we will choose cubes on the diagonal consecutively until we hit the convex hull of the cross. For example, at the origin, we choose cubes on the diagonal until we hit the hyperplane $x_1+...+x_d = \frac{(n-1)(d-1)}{2}+1$, which is the convex hull of the points $(1,\frac{n-1}{2}, ...,\frac{n-1}{2})$,$(\frac{n-1}{2},1, \frac{n-1}{2} ...,\frac{n-1}{2})$,..., $(\frac{n-1}{2}, ...,\frac{n-1}{2},1)$. As these diagonal cubes are of the form $(j,j,...,j)$ where $j = 1,...,n-1$, When we plug them into the hyperplane equation, we find that 
\begin{equation}\label{eq_Nd}
N_d = \lfloor\frac{(n-1)(d-1)}{2d}+\frac1{d}\rfloor
\end{equation}
cubes are needed. By symmetry, we choose, from every corner of the cubes, $N_d$ consecutive diagonal cubes from the corner. Collect all those digits as the set $\mathtt{Corner}$. 
     Then our digit set is ${\mathcal D} =\mathtt{Cross}\cup \mathtt{Corner} $,
     and our IFS will be $\{\frac{1}{n} (x+d) : d \in \mathcal{D}\}$. In Figure \ref{Fig_a2d}, the first two iterations of the 2 dimensional cross and corner fractal squares are shown. In Figure \ref{Fig}, the cross and one side of the corners of the fractal construction is shown. 

      \begin{figure}[h]\label{Fig_a2d}
            \includegraphics[height=2.5in]{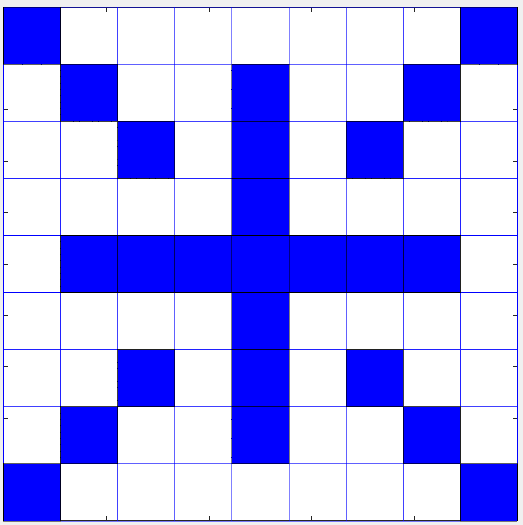}  
                        \includegraphics[height=2.5in]{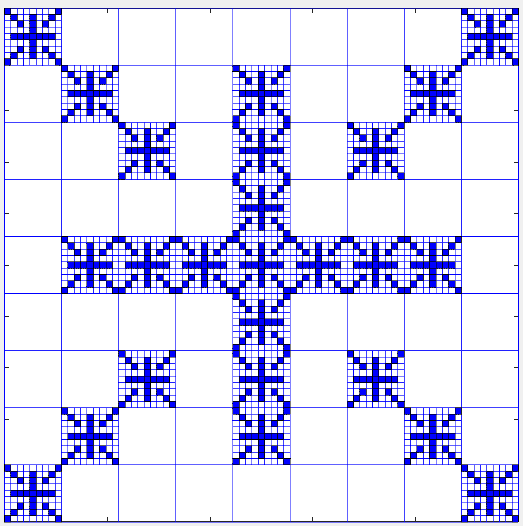} 
         \caption{A two dimensional picture for the cross and corner fractal}
                \end{figure}

      \begin{figure}[h]\label{Fig}
            \includegraphics[height=3in]{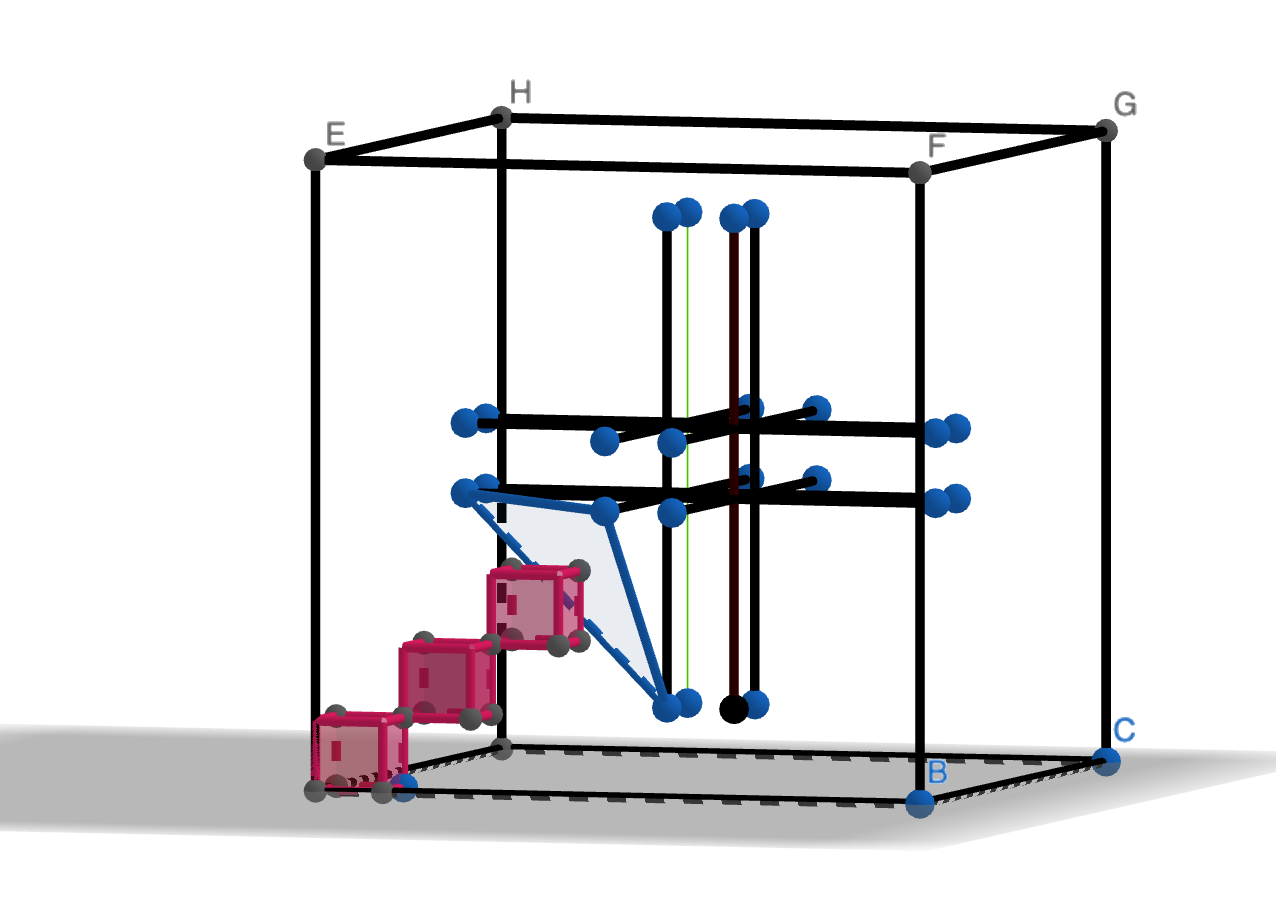}  
    \caption{A part of the three dimensional picture for the cross and corner fractal. The whole construction is to include the diagonal cubes starting from all corners}
                \end{figure}
     We will call $K_n$ to be the self-similar fractal cubes that this IFS generates. Then, \\ (1). $K_n$ is totally disconnected. \\
     (2). $K_n$ projects very thick shadows in every 1-dimensional subspace. \\
     (3). dim$_H(K_n)\to 1$ as $n\to\infty$. 
\end{example}
     \medskip
     \begin{proof}

(1). To see that $K_n$ is totally disconnected, we apply Proposition \ref{prop_totally-disconnect}. At the $k$th stage of the iteration, the connected components are either a cross inside a cube of length $1/k$ or a union of corner cubes which lies in $2^d$ cubes meeting at a point. As these union of corner cubes does not intersect the cross, so this component has diameter at most $\sqrt{2d}/k$. Therefore, the maximum of the diameters of the connected components at the $k$th iteration is bounded above by  $\sqrt{2d}/k$, which tends to 0 as $k\to \infty$. This shows $K_n$ is totally disconnected. 

     \medskip
     
    (2).     We now show that this fractal projects to intervals for all lines on $\R^d$. Note that by our construction, 
 the convex hull of the cross must intersect the convex hull of the connected components at each corner. Therefore, it is not possible to separate any connected components by hyperplanes. By Theorem \ref{classification theorem2}, our proof is complete. 

\medskip

(3).  For a fixed odd integer $n>2$, there are $d(n-2)+ d N_d $ number of cubes chosen for this pattern where $N_d$ is given by (\ref{eq_Nd}). Hence,  
$$
\mbox{dim}_H(K_n) = \frac{\log (d(n-2)+dN_d)}{\log n}.
$$
 As $d(n-2)+ d N_d  = O(n)$, the above expression goes to 1 as $n\to\infty$. 
     \end{proof}

\medskip

\section{Fractal Decomposability for other convex sets}

In the previous section, we demonstrate that cubes on $\R^d$ are fractal decomposable with self-similar sets of dimension arbitrarily close to one.  We will now extend our study to other convex polytopes in this section. Let $\overrightarrow{AB}$ denote the line segment joining $A$ and $B$ with direction pointing from point A to point B. 
A simplex on $\R^d$ is the convex hull of $d+1$ points $\{A_0,...,A_{d}\}$ where the vectors $\{\overrightarrow{A_0A_1},...,\overrightarrow{A_0A_d}\}$ are linearly independent. We will denote it by ${\mathcal S} (A_0,...,A_d)$

\medskip

Given a simplex ${\mathcal S} ={\mathcal S} (A_0,...,A_d)$, we consider another similar unrotated image of ${\mathcal S}$, denoted by ${\mathcal S} (A_0',...,A_d')  = r{\mathcal S} (A_0,...,A_d)+t$, where $0<r<1$ and $t$ is a translation vector such that ${\mathcal S} (A_0',...,A_d')$ lies in the interior of ${\mathcal S} (A_0,...,A_d)$.

\medskip

Fix $\lambda>0$. Notice that for all $i = 0,1,...,d$, the set $\{\overrightarrow{A_iA_j}: j\ne i, j\in \{0,...,d\}\}$ forms a basis for $\R^d$. We now define $d+1$ affine maps inside ${\mathtt S}$ by the following relations. For $i = 0,1,...,d$ and $j = 0,1,...,d$, 
$$
T_i (\overrightarrow{A_iA_j}) = \lambda \overrightarrow{A_iA_j}, \ \ j \ne i-1, \text{and}
$$
\begin{equation}\label{eqT_construct}
    T_i (\overrightarrow{A_iA_{i-1}}) =\overrightarrow{A_iA_{i+1}'}, \ \  j = i-1. 
\end{equation} 
 We will identify the addition as the addition on the cyclic group of $d+1$ elements. i.e. If $i=0$, we identify $ -1 = d$.  $\Phi_{\lambda} = \{T_0,..., T_d\}$ defines a self-affine IFS whose attractor has convex hull exactly equal to ${\mathcal S}(A_0,...,A_d).$ $A_i$ is the fixed point of $T_i$ for each $i$.  In this IFS, 
 $$
 T_i({\mathcal S}(A_0,...,A_d)) = {\mathcal S} (A_i, \{A_i+\lambda\overrightarrow{A_iA_{j}}:j\ne i-1,i\},{A_{i+1}'}).
 $$
Figure \ref{fig:triangular-ifs-iterations} and \ref{fig:R3_simplex} illustrate the IFS on $\R^2$ and $\R^3$. 

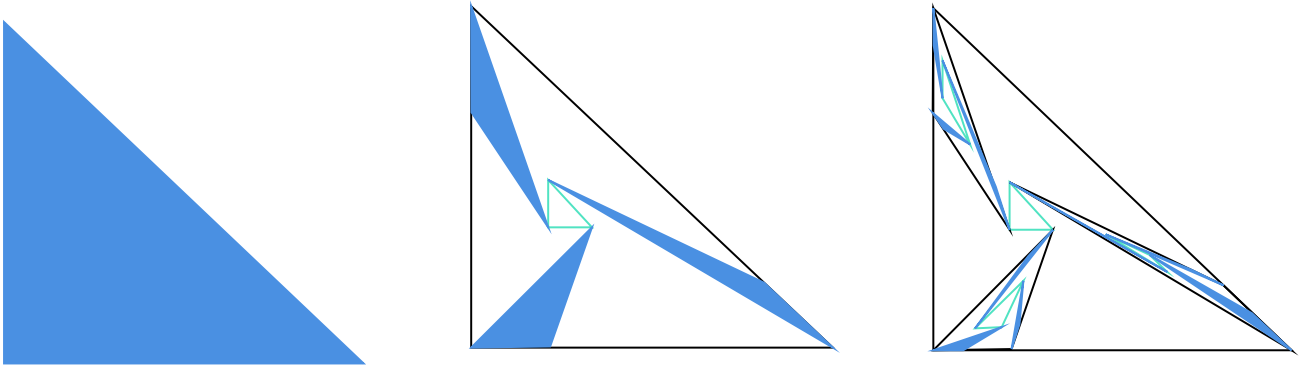
\begin{figure}[h]
    \centering

\tikzset{every picture/.style={line width=0.75pt}} 

\begin{tikzpicture}[x=0.75pt,y=0.75pt,yscale=-1,xscale=1]

\draw  [color={rgb, 255:red, 74; green, 144; blue, 226 }  ,draw opacity=1 ][fill={rgb, 255:red, 74; green, 144; blue, 226 }  ,fill opacity=1 ] (7,44) -- (188.34,216.33) -- (7,216.33) -- cycle ;
\draw   (242.6,36) -- (425.02,208.33) -- (242.6,208.33) -- cycle ;
\draw  [color={rgb, 255:red, 80; green, 227; blue, 194 }  ,draw opacity=1 ] (281.43,123.7) -- (303.52,147.62) -- (281.43,147.62) -- cycle ;
\draw  [color={rgb, 255:red, 74; green, 144; blue, 226 }  ,draw opacity=1 ][fill={rgb, 255:red, 74; green, 144; blue, 226 }  ,fill opacity=1 ] (303.52,147.62) -- (282.43,207.72) -- (242.6,208.33) -- cycle ;
\draw  [color={rgb, 255:red, 74; green, 144; blue, 226 }  ,draw opacity=1 ][fill={rgb, 255:red, 74; green, 144; blue, 226 }  ,fill opacity=1 ] (281.43,147.62) -- (242.77,89.56) -- (242.6,36) -- cycle ;
\draw  [color={rgb, 255:red, 74; green, 144; blue, 226 }  ,draw opacity=1 ][fill={rgb, 255:red, 74; green, 144; blue, 226 }  ,fill opacity=1 ] (281.43,123.7) -- (425.02,208.33) -- (390.37,175.62) -- cycle ;
\draw   (475.68,37) -- (656.33,209.67) -- (475.68,209.67) -- cycle ;
\draw  [color={rgb, 255:red, 80; green, 227; blue, 194 }  ,draw opacity=1 ] (514.13,124.87) -- (536.01,148.83) -- (514.13,148.83) -- cycle ;
\draw   (536.01,148.83) -- (515.13,209.05) -- (475.68,209.67) -- cycle ;
\draw  [color={rgb, 255:red, 80; green, 227; blue, 194 }  ,draw opacity=1 ] (521.09,174.64) -- (510.15,197.99) -- (496.73,198.61) -- cycle ;
\draw   (514.13,148.83) -- (475.35,89.84) -- (475.68,37) -- cycle ;
\draw   (514.13,124.87) -- (656.33,209.67) -- (622.03,176.89) -- cycle ;
\draw  [color={rgb, 255:red, 80; green, 227; blue, 194 }  ,draw opacity=1 ] (494.24,105.82) -- (480.32,82.47) -- (480.32,63.42) -- cycle ;
\draw  [color={rgb, 255:red, 74; green, 144; blue, 226 }  ,draw opacity=1 ][fill={rgb, 255:red, 74; green, 144; blue, 226 }  ,fill opacity=1 ] (510.15,197.99) -- (490.76,209.67) -- (475.68,209.67) -- cycle ;
\draw  [color={rgb, 255:red, 74; green, 144; blue, 226 }  ,draw opacity=1 ][fill={rgb, 255:red, 74; green, 144; blue, 226 }  ,fill opacity=1 ] (494.24,105.82) -- (480.82,97.83) -- (475.35,89.84) -- cycle ;
\draw  [color={rgb, 255:red, 74; green, 144; blue, 226 }  ,draw opacity=1 ][fill={rgb, 255:red, 74; green, 144; blue, 226 }  ,fill opacity=1 ] (496.73,198.61) -- (523.58,162.15) -- (536.01,148.83) -- cycle ;
\draw  [color={rgb, 255:red, 74; green, 144; blue, 226 }  ,draw opacity=1 ][fill={rgb, 255:red, 74; green, 144; blue, 226 }  ,fill opacity=1 ] (475.85,55.84) -- (480.32,82.47) -- (475.68,37) -- cycle ;
\draw  [color={rgb, 255:red, 74; green, 144; blue, 226 }  ,draw opacity=1 ][fill={rgb, 255:red, 74; green, 144; blue, 226 }  ,fill opacity=1 ] (521.09,174.64) -- (519.1,196.56) -- (515.13,209.05) -- cycle ;
\draw  [color={rgb, 255:red, 74; green, 144; blue, 226 }  ,draw opacity=1 ][fill={rgb, 255:red, 74; green, 144; blue, 226 }  ,fill opacity=1 ] (480.32,63.42) -- (507.17,127.12) -- (514.13,148.83) -- cycle ;
\draw  [color={rgb, 255:red, 80; green, 227; blue, 194 }  ,draw opacity=1 ] (562.61,151.19) -- (593.68,170.14) -- (584.73,161.53) -- cycle ;
\draw  [color={rgb, 255:red, 74; green, 144; blue, 226 }  ,draw opacity=1 ][fill={rgb, 255:red, 74; green, 144; blue, 226 }  ,fill opacity=1 ] (584.73,161.53) -- (656.33,209.67) -- (636.44,191.03) -- cycle ;
\draw  [color={rgb, 255:red, 74; green, 144; blue, 226 }  ,draw opacity=1 ][fill={rgb, 255:red, 74; green, 144; blue, 226 }  ,fill opacity=1 ] (562.61,151.19) -- (622.03,176.89) -- (605.62,168.91) -- cycle ;
\draw  [color={rgb, 255:red, 74; green, 144; blue, 226 }  ,draw opacity=1 ][fill={rgb, 255:red, 6; green, 34; blue, 70 }  ,fill opacity=1 ] (593.68,170.14) -- (537,138.18) -- (514.13,124.87) -- cycle ;

\end{tikzpicture}
    \caption{Illustration of the triangular IFS on $\R^2$.}
    \label{fig:triangular-ifs-iterations}
\end{figure}

\begin{figure}[h]
    \centering

\begin{tikzpicture}[x=0.75pt,y=0.75pt,yscale=-1,xscale=1]

\draw   (351.38,260.79) -- (365.79,10.57) -- (100.33,202.56) -- cycle ;
\draw  [dash pattern={on 4.5pt off 4.5pt}]  (100.33,202.56) -- (578.33,160.47) ;
\draw   (578.33,160.47) -- (351.38,260.79) -- (365.79,10.57) -- cycle ;
\draw   (357.78,173.15) -- (361.75,97.05) -- (288.62,155.44) -- cycle ;
\draw  [dash pattern={on 4.5pt off 4.5pt}]  (288.62,155.44) -- (420.31,142.64) ;
\draw   (420.31,142.64) -- (357.78,173.15) -- (361.75,97.05) -- cycle ;
\draw  [color={rgb, 255:red, 208; green, 2; blue, 27 }  ,draw opacity=1 ][fill={rgb, 255:red, 208; green, 2; blue, 27 }  ,fill opacity=0.25 ] (357.78,173.15) -- (168.14,217.74) -- (100.33,202.56) -- cycle ;
\draw [color={rgb, 255:red, 208; green, 2; blue, 27 }  ,draw opacity=1 ] [dash pattern={on 4.5pt off 4.5pt}]  (100.33,202.56) -- (222,191) ;
\draw [color={rgb, 255:red, 208; green, 2; blue, 27 }  ,draw opacity=1 ] [dash pattern={on 4.5pt off 4.5pt}]  (365.79,10.57) ;
\draw    (365.79,10.57) -- (351.38,260.79) ;
\draw  [color={rgb, 255:red, 237; green, 130; blue, 130 }  ,draw opacity=1 ][fill={rgb, 255:red, 254; green, 189; blue, 189 }  ,fill opacity=1 ] (365.79,10.57) -- (363,37) -- (288.62,155.44) -- (288.62,155.44) -- (338,30) -- cycle ;
\draw [color={rgb, 255:red, 225; green, 132; blue, 132 }  ,draw opacity=1 ]   (338,30) -- (363,37) ;
\draw [color={rgb, 255:red, 225; green, 124; blue, 124 }  ,draw opacity=1 ] [dash pattern={on 0.84pt off 2.51pt}]  (365.79,10.57) -- (292,148) ;
\draw [color={rgb, 255:red, 239; green, 140; blue, 140 }  ,draw opacity=1 ] [dash pattern={on 4.5pt off 4.5pt}]  (172,195) -- (168.14,217.74) ;
\draw  [color={rgb, 255:red, 235; green, 152; blue, 152 }  ,draw opacity=1 ][fill={rgb, 255:red, 234; green, 181; blue, 181 }  ,fill opacity=1 ] (578.33,160.47) -- (578.33,160.47) -- (516,164) -- (361.75,97.05) -- (578.33,160.47) -- cycle ;
\draw  [color={rgb, 255:red, 235; green, 152; blue, 152 }  ,draw opacity=1 ][fill={rgb, 255:red, 234; green, 181; blue, 181 }  ,fill opacity=1 ] (578.33,160.47) -- (578.33,160.47) -- (535,129) -- (361.75,97.05) -- (578.33,160.47) -- cycle ;
\draw [color={rgb, 255:red, 213; green, 134; blue, 134 }  ,draw opacity=1 ] [dash pattern={on 4.5pt off 4.5pt}]  (535,129) -- (521,167) ;
\draw  [color={rgb, 255:red, 224; green, 164; blue, 164 }  ,draw opacity=1 ][fill={rgb, 255:red, 232; green, 194; blue, 194 }  ,fill opacity=1 ] (378,248) -- (378,248) -- (351.38,260.79) -- (352,236) -- (420.31,142.64) -- cycle ;
\draw [color={rgb, 255:red, 210; green, 136; blue, 136 }  ,draw opacity=1 ]   (420.31,142.64) -- (351.38,260.79) ;

\draw (78,193.8) node [anchor=north west][inner sep=0.75pt]   [align=left] {$\displaystyle A_{1}$};
\draw (395,1.57) node [anchor=north west][inner sep=0.75pt]   [align=left] {$\displaystyle A_{4}$};
\draw (584,154.39) node [anchor=north west][inner sep=0.75pt]   [align=left] {$\displaystyle A_{3}$};
\draw (340,263) node [anchor=north west][inner sep=0.75pt]   [align=left] {$\displaystyle A_{2}$};
\draw (261,147.8) node [anchor=north west][inner sep=0.75pt]   [align=left] {$\displaystyle A_{1} '$};
\draw (335,82.8) node [anchor=north west][inner sep=0.75pt]   [align=left] {$\displaystyle A_{4} '$};
\draw (423,130.8) node [anchor=north west][inner sep=0.75pt]   [align=left] {$\displaystyle A_{3} '$};
\draw (355,168.8) node [anchor=north west][inner sep=0.75pt]   [align=left] {$\displaystyle A_{2} '$};
\draw (142,128) node [anchor=north west][inner sep=0.75pt]   [align=left] {};
\end{tikzpicture}
    \caption{An Illustration of the IFS on the simplex on $\R^3$}
    \label{fig:R3_simplex}
\end{figure}
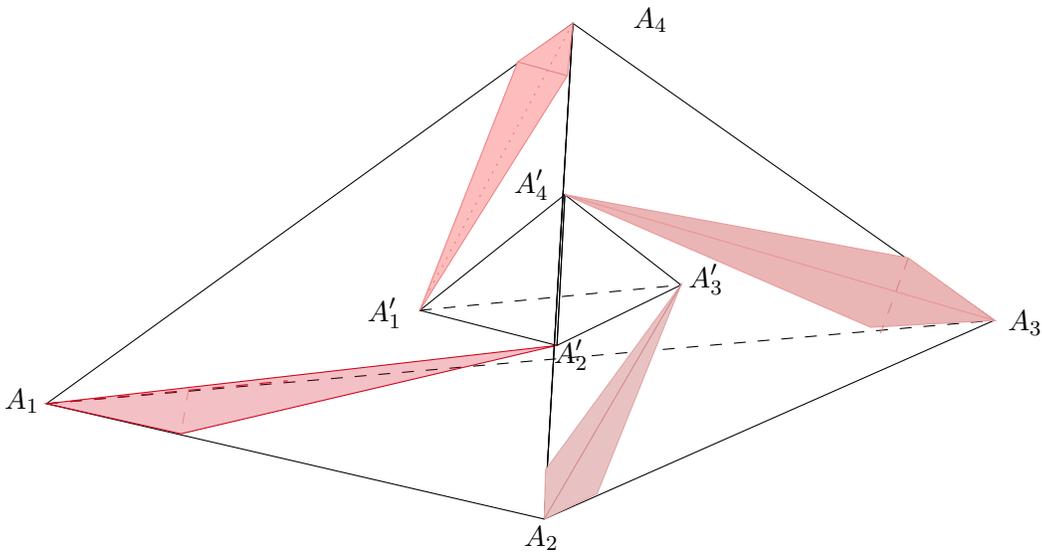

\begin{theorem}\label{theorem_simplex}
    The attractor $K_{\lambda}$ of the IFS $\Phi_{\lambda}$ defined above is totally disconnected, projects very thick shadows in every 1-dimensional subspace, and has dim$_H(K_{\lambda})\to 1$ as $\lambda\to 0.$

    \medskip

    Consequently, all simplices on $\R^d$ are fractal-decomposable with self-affine sets whose Hausdorff dimension is arbitrarily close to 1. 
\end{theorem}

    As $T_i({\mathcal S}(A_0,...,A_d))$ are mutually disjoint, it follows that the self-affine set must be totally disconnected. In Proposition \ref{prop_thick_simplices}, we will verify condition (2) from Theorem \ref{classification theorem2}. Hence, the invariant set will project very thick shadows in every 1-dimensional subspace. Finally, we will show that the Hausdorff dimension can be made arbitrarily close to 1 in Proposition \ref{prop_Hdim}.

\medskip

 \begin{proposition}\label{prop_thick_simplices}
    Let ${\mathcal J}_i = T_i({\mathcal S}(A_0,...,A_d)) $. Then for all  $I\subsetneq \{0,1,...,d\}$, 
    $$
    \mbox{conv} \left(\bigcup_{i\in I}{\mathcal J}_i\right)\cap\mbox{conv} \left(\bigcup_{i\in I^c}{\mathcal J}_i\right) \ne \emptyset.
    $$
 \end{proposition}

\begin{proof}
    Let $I\subsetneq \{0,1,...,d\}$ and let $i\in I$. We first notice that if $i\in I$, then the line segment $\overrightarrow{A_iA_{i+1}'}\in  \mbox{conv} \left(\bigcup_{i\in I}{\mathcal J}_i\right)$ by (\ref{eqT_construct}). The proof will be based on a simple geometrical fact.

\medskip

    {\bf Fact:} Let $\overrightarrow{AB}$ and $\overrightarrow{CD}$ be two parallel line segments on $\R^2$ such that the vectors are in the same direction. Then the line segment $\overrightarrow{AD}$ and $\overrightarrow{CB}$ must intersect.
    
    \medskip
    
   The proof of this fact is elementary (e.g. we can put  $\overrightarrow{AB}$ and $\overrightarrow{CD}$ on $y=0$ and $y=1$ respectively and solve for the intersection), so we will omit it. We now have two cases:

    
    \medskip
    
   (1). Suppose that $i-1,i+1\in I^c.$ We note that the line segment $\overrightarrow{A_iA_{i+1}'}\in \mbox{conv} \left(\bigcup_{i\in I}{\mathcal J}_i\right)$. On the other hand, the points $A_i'$ and $A_{i+1}$ are in $\mbox{conv} \left(\bigcup_{i\in I^c}{\mathcal J}_i\right)$. Hence, the line segment $\overrightarrow{A_{i}'A_{i+1}}\in\mbox{conv} \left(\bigcup_{i\in I^c}{\mathcal J}_i\right) $ as well. From our construction of the smaller simplex,  $\overrightarrow{A_iA_{i+1}}$ and $\overrightarrow{A_i'A_{i+1}'}$ are two parallel line segments pointing in the same direction. Thus, the points $A_i,A_i',A_{i+1},A_{i+1}'$ lie in the same two dimensional plane. In particular, $\overrightarrow{A_{i}A_{i+1}'}$ and $\overrightarrow{A_{i}'A_{i+1}}$ must intersect by the {\bf Fact}. This means that the convex hull cannot be empty.

   \medskip

    Because of (1), if $I$ consists of only one element, then the conclusion of the proposition must hold. The same is true if $I^c$ consists of only one element as we can switch the role of $I$ and $I^c$ in the proof of case (1). Hence, we can assume that both $I$ and $I^c$ contains at least two elements for case (2).

   \medskip

   (2). If (1) does not hold and $I,I^c$ has at least two elements, then $i-1$ or $i+1\in I.$ In this case, we let 
   $$
   j = \min\{i'\ge i: i'\in I, i'+1\in I^c\}, \ k = \max\{i'< i: i'+1\in I, i'\in I^c\},
   $$   
   Since $I,I^c$ has at least two elements and (1) does not hold,  $k+1\ne j$ and $j+1\ne k$. Moreover, we have $j\in I$, $j+1\in I^c$, $k\in I^c$ and   $k+1\in I$.  This means that the points $A_{j+1}', A_{k+1}\in \mbox{conv}(\bigcup_{i\in I}{\mathcal J}_i)$ and $A_{j+1}, A_{k+1}'\in \mbox{conv}(\bigcup_{i\in I^c}{\mathcal J}_i)$, and so are their respective line segments. We now consider the plane generated by $A_{j+1}, A_{j+1}', A_{k+1}, A_{k+1}'$. The {\bf Fact} implies that the line segment $\overrightarrow{A_{j+1}A_{k+1}'}$ and $\overrightarrow{A_{j+1}'A_{k+1}}$  must intersect, so the intersection of the convex hulls cannot be empty. 
 This completes the whole proof. 
   \end{proof}

\medskip

 \begin{proposition}\label{prop_Hdim}
$$\lim_{\lambda\to0} \dim_H(K_{\lambda})= 1.$$
 \end{proposition}
 \begin{proof}
 We will give an estimate of the affinity dimension (\ref{affinity}) of the affine IFS . To do so, we will need to give an estimate of the singular values of the linear part of $T_i$. We will let $\alpha_{i,1}\ge\alpha_{i,2}...\ge\alpha_{i,d}>0$ be the singular values of $T_i$. Also, let
$$
\alpha_1 = \max \{\alpha_{i,1}: i = 0,...,d\},\ \alpha_2 = \max \{\alpha_{i,2}: i = 0,...,d\}.
$$
As $A_i$ is the fixed point of $T_i$, by putting $A_i$ as the origin of the coordinates, we may assume $T_i$ is a linear transformation. Moreover,  the subspace $H_i$ containing the face of the simplex generated by the points $\{A_j: j\in \{0,1,...,d\}\setminus\{i-1\}\}$ is the eigenspace of $T_i$ with eigenvalue $\lambda.$ Let ${\mathcal B}_i$ be an orthonormal basis for $H_i$; we extend ${\mathcal B}_i$ to an orthonormal basis for $\R^d$ by adding one more vector. Then 
$T$ admits a matrix representation of the form $\begin{bmatrix}\lambda I_{d-1} & {\bf u} \\ {\bf 0} & u_d \end{bmatrix}$ where $I_{d-1}$ is the $(d-1) \times$ \st{by} $(d-1)$ identity matrix and ${\bf u} = (u_1,...,u_{d-1})^{\top}$. Hence, 
$$
T^{\ast} T = \begin{bmatrix}\lambda^2 I_{d-1} & \lambda{\bf u} \\ \lambda{\bf u} & \sum_{i=1}^d u_i^2
\end{bmatrix}.
$$
By the Cauchy interlacing theorem applied to the principal minor $\lambda^2 I_{d-1}$, we have $\alpha_2\le \lambda^2$ (they are indeed equal if $d\ge 3$). Recall that the singular value function is sub-multiplicative. Therefore,
    \begin{align*}
        \sum_{k=1}^\infty \sum_{\sigma \in \Sigma^k} \varphi^s(T_\sigma) 
        &\leq \sum_{k=1}^\infty \sum_{\sigma \in \Sigma^k} (\alpha_{1, \sigma_1}\alpha_{2, \sigma_1}^{s-1})\cdots (\alpha_{1, \sigma_k} \alpha_{2, \sigma_k}^{s-1}) \\
        &\leq \sum_{k=1}^\infty \sum_{\sigma \in \Sigma^k} \alpha_1^k (\alpha_2^{s-1})^k \\
        &= \sum_{k=1}^\infty ((d+1)\alpha_1 \alpha_2^{s-1})^k
    \end{align*}
    The above summation is finite if and only if $(d+1)\alpha_1 \alpha_2^{s-1} < 1$. So we must have
    \[\dim_a(K) \leq 1+ \frac{\log((d+1)\alpha_1)}{\log\alpha_2^{-1}}.\]
    Since $\alpha_2 \to 0$ as  $\lambda \to 0$, we must have  $\dim_H(K_{\lambda})\le \dim_a(K) \to 1$.  Proposition \ref{prop_thick_simplices} implies that $K_\lambda$ projects very thick shadows, so $\dim_H(K_{\lambda})\ge 1$. Hence, our proposition follows. 
 \end{proof}

\noindent{\bf Proof of Theorem \ref{theorem_Main1} (1).} We note that if $C$ is a convex polytope,, then $C$ is a finite union of simplices. Denote all these simplices by  $\{\Delta_1, \Delta_2, \ldots, \Delta_N\}$. By Theorem \ref{theorem_simplex}, there exists an affine IFS $\Phi_j$ whose attractor $K_j$ projects very thick shadows in every 1-dimensional subspace and has Hausdorff dimension aribtirarily close to 1\st{,} for each $\Delta_j$, where $1 \leq j \leq N$. By Proposition \ref{prop_totally-disconnect}, the union 
    \[K = \bigcup_{j=1}^N K_j\]
    is totally disconnected. Moreover, $K$ is perfect since each $K_j$ is perfect and the finite union of perfect sets is perfect. $K$ also has Hausdorff dimension arbitrarily close to 1 by the countable stability of Hausdorff dimension.
    It remains to show that $\pi_W(C) = \pi_W(K)$ for all $W\in G(d,1)$. Consider the following: 
    \begin{align*}
        \pi_W(C) &= \pi_W \left( \bigcup_{j=1}^N \Delta_j \right) 
        = \bigcup_{j=1}^N \pi_W(\Delta_j)
        = \bigcup_{j=1}^N \pi_W(K_j) 
        = \pi_W \left( \bigcup_{j=1}^N K_j \right) 
        = \pi_W (K).
    \end{align*}
    So every convex polytope is fractal decomposable with a finite union of self-affine sets whose Hausdorff dimension is arbitrarily close to 1, completing the proof of Theorem \ref{theorem_Main1} (1).

\medskip

{\bf Remark.} It is possible to obtain a very thin self-similar set inside  any triangle on $\R^2$ projecting very thick shadows in every direction. To see this, let $\Delta$ be a triangle. We notice a simple fact that every triangle is self-similar with 4 maps of contraction ratio 1/2. For $\lambda>0$ given in the affine IFS construction (See Figure \ref{fig:triangular-ifs-iterations}), we now partition the triangle into $4^n$  similar triangles where  $n = n_{\lambda}$ is the integer such that 
$$
 2^{-n}\le \frac{\lambda}{100} <2^{-n+1}.
$$
We now take all those small triangles that intersect that intersect $\Delta_i = T_i(\Delta)$, for $i=1,2,3$. These triangles form the first iteration of the self-similar IFS. Call this IFS $\Phi$. Because the self-similar triangulation covers each $\Delta_i$, the convex hull condition of Theorem \ref{classification theorem2} still holds, ensuring we have very thick shadows in every direction. Additionally, because there is a small triangle separating $\Delta_i$, the diameter of each connected component goes to 0, and we have that the attractor $K$ of $\Phi$ is totally disconnected by Theorem \ref{0-diam-total-disconn}.

Finally, we will calculate the Hausdorff dimension of $K$. We know that the contraction ratio is $\frac{1}{2^n}$. By standard volume counting, we estimate the number of maps covering each $\Delta_i$ is  $O(2^{2n} \lambda)$. Therefore, for some universal constant $C>0$,
\[\dim_H(K) \le  \frac{\log (C 2^{2n} \lambda)}{\log 2^n}  = 1+ \frac{\log (C2^n\lambda)}{\log 2^n} < 1+ \frac{\log 200 C}{\log 2^n}.\]
Since $n \to \infty$ as $\lambda \to 0$, we see that $\dim_H(K) \to 1$ as $\lambda \to 0$.

\medskip

\section{A Compact Fractal Decomposable Rectifiable Set}\label{Section6}

In this section, we aim at constructing a rectifiable compact set that is fractal decomposable. Let us recall some terminologies. A \textbf{$1$-set} is a set $E$ with finite and positive one-dimensional Hausdorff measure. A $1$-set $E$ is called {\bf rectifiable} if there exists countable sets $A_i$ and Lipschitz functions $f_i$ such that 
$$
{\mathcal H}^1 \left(E\setminus \bigcup_{i=1}^{\infty} f_i(A_i)\right) = 0.
$$
$E$ is called {\bf purely unrectifiable} if ${\mathcal H}^1(E\cap F) = 0$ for all rectifiable sets $F$. The well-known Bescovitch projection theorem states that 

\begin{theorem} [Bescovitch Projection Theorem, see \cite{Mattila_1995}]
    For a purely unrectifiable 1-set on $\R^2$, the projection of $E$ must have measure zero  for almost all directions.
\end{theorem}

 Therefore purely unrectifiable sets will not project\st{ to} very thick shadows in almost all directions.

\medskip

 we will now prove Theorem \ref{theorem_Main1} (2), which predicts that every convex set is fractal-decomposable with compact totally disconnected rectifiable 1-sets. This question was first brought to Alan Chang by the first-named author. He later on discussed with Tuomas Oprenon, who provided to us a workable idea of the construction. We would like to thank Alan Chang and Tuomas Oprenon for supplying the main idea of the proof in the key lemma below, which utilizes the Venetian Blind Construction. Since the lemma is on dimension 2, we will parametrize the projection by $\pi_{\theta}$, $\theta\in[0,\pi)$, where $\pi_{\theta}$ is the orthogonal projection onto the line $y = (\tan\theta) x$. 

\begin{lemma} \label{tot-discon-rect-set-lemma}
    Let $E_0 = [0, 1] \times \{0\}$. Then there exists a totally disconnected rectifiable compact 1-set $E \subset [0, 1]^2$ such that $\pi_\theta (E) \supset \pi_\theta(E_0)$ for $\theta \in [0, \pi)$.
\end{lemma}
\begin{proof}
    It suffices to construct $E$ such that the conclusion holds for $\theta \in [0, \frac{\pi}{2}]$. To extend the conclusion to $\theta \in [\frac{\pi}{2}, \pi)$, we need only reflect $E$ about the line $x = \frac{1}{2}$. Unioning $E$ and its reflection will give us our desired result. 

    Now we will construct $E$ for $\theta \in [0, \frac{\pi}{2}]$ using the Venetian blind construction\footnote{Another example of this construction can be found on page 104 of \cite{falconer2004fractal}.}. To begin, $L(\mathbf{a}, \mathbf{b})$ will represent the line segment connecting $\mathbf{a}, \mathbf{b} \in \R^2$. Let $\{\varepsilon_n\}$ be a rapidly decreasing sequence, which will be specified later. Let
    \[E_1 = L\left((0, 0), \; \left(\frac{1}{2}, \varepsilon_1\right)\right) \cup L\left(\left( \frac{1}{2}, 0\right), \; \left(\frac{1}{2}, 1+\varepsilon_1\right)\right). \]
    Suppose $E_i$ is constructed. Let ${\bf e_2} = (0,1)$.
    Then we define 
    \[E_{i+1} = \bigcup_{L(\mathbf{a}, \mathbf{b}) \subset E_i} \left( L\left(\mathbf{a}, \; \frac{\mathbf{a} + \mathbf{b}}{2} + \varepsilon_{i+1}\mathbf{e_2}\right) \cup L\left(\frac{\mathbf{a} + \mathbf{b}}{2}, \mathbf{b} + \varepsilon_{i+1}\mathbf{e_2}\right)\right).\]
    There are $2^{j}$ line segments in $E_j$; for simplicity, we will write $E_j = \bigcup_{k=1}^{2^j} L_{k,j}$, where $L_{k,j}$ is the line segment ending at the vertical line $x = \frac{k}{2^j}$. We can think of $L_{k,j}$ as the    ``blind".  This construction of $E_i$ is shown in Figure \ref{E_n-construction}. Recall the Hausdorff metric between two compact sets is defined as
    $$
    d_H(E,F)  = \inf \left\{\delta: E\subset (F)_{\delta} \ \mbox{and} \ F\subset (E)_{\delta}\right\},
    $$
    where $(E)_{\delta}$ is the $\delta$-neighborhood of the set $E$. It is well-known that the set of all compact sets forms a complete metric space under $d_H$.

    \begin{figure}
        \centering
\tikzset{every picture/.style={line width=0.75pt}} 

\begin{tikzpicture}[x=0.75pt,y=0.75pt,yscale=-1,xscale=1]

\draw [color={rgb, 255:red, 124; green, 31; blue, 197 }  ,draw opacity=1 ][line width=2.25]    (40,20) -- (560,20) ;
\draw [color={rgb, 255:red, 0; green, 0; blue, 0 }  ,draw opacity=1 ][line width=0.75]    (40,130) -- (560,130) ;
\draw [color={rgb, 255:red, 124; green, 31; blue, 197 }  ,draw opacity=1 ][line width=2.25]    (40,130) -- (300,70) ;
\draw  [dash pattern={on 4.5pt off 4.5pt}]  (300,70) -- (300,130) ;
\draw  [dash pattern={on 4.5pt off 4.5pt}]  (560,70) -- (560,130) ;
\draw [color={rgb, 255:red, 124; green, 31; blue, 197 }  ,draw opacity=1 ][line width=2.25]    (300,130) -- (560,70) ;
\draw [color={rgb, 255:red, 0; green, 0; blue, 0 }  ,draw opacity=1 ][line width=0.75]    (40,270) -- (560,270) ;
\draw [color={rgb, 255:red, 0; green, 0; blue, 0 }  ,draw opacity=1 ][line width=0.75]  [dash pattern={on 4.5pt off 4.5pt}]  (40,270) -- (300,210) ;
\draw  [dash pattern={on 4.5pt off 4.5pt}]  (300,210) -- (300,270) ;
\draw  [dash pattern={on 4.5pt off 4.5pt}]  (560,210) -- (560,270) ;
\draw [color={rgb, 255:red, 0; green, 0; blue, 0 }  ,draw opacity=1 ][line width=0.75]  [dash pattern={on 4.5pt off 4.5pt}]  (300,270) -- (560,210) ;
\draw [color={rgb, 255:red, 124; green, 31; blue, 197 }  ,draw opacity=1 ][line width=2.25]    (40,270) -- (170,210) ;
\draw [color={rgb, 255:red, 124; green, 31; blue, 197 }  ,draw opacity=1 ][line width=2.25]    (170,240) -- (300,180) ;
\draw [color={rgb, 255:red, 124; green, 31; blue, 197 }  ,draw opacity=1 ][line width=2.25]    (430,240) -- (560,180) ;
\draw [color={rgb, 255:red, 124; green, 31; blue, 197 }  ,draw opacity=1 ][line width=2.25]    (300,270) -- (430,210) ;
\draw  [dash pattern={on 4.5pt off 4.5pt}]  (560,180) -- (560,210) ;
\draw  [dash pattern={on 4.5pt off 4.5pt}]  (430,210) -- (430,240) ;
\draw  [dash pattern={on 4.5pt off 4.5pt}]  (300,180) -- (300,210) ;
\draw  [dash pattern={on 4.5pt off 4.5pt}]  (170,210) -- (170,240) ;
\draw    (570,70) -- (570,97) -- (570,130) ;
\draw [shift={(570,130)}, rotate = 270] [color={rgb, 255:red, 0; green, 0; blue, 0 }  ][line width=0.75]    (0,5.59) -- (0,-5.59)   ;
\draw [shift={(570,70)}, rotate = 270] [color={rgb, 255:red, 0; green, 0; blue, 0 }  ][line width=0.75]    (0,5.59) -- (0,-5.59)   ;
\draw    (570,210) -- (570,237) -- (570,270) ;
\draw [shift={(570,270)}, rotate = 270] [color={rgb, 255:red, 0; green, 0; blue, 0 }  ][line width=0.75]    (0,5.59) -- (0,-5.59)   ;
\draw [shift={(570,210)}, rotate = 270] [color={rgb, 255:red, 0; green, 0; blue, 0 }  ][line width=0.75]    (0,5.59) -- (0,-5.59)   ;
\draw    (570,180) -- (570,207) -- (570,210) ;
\draw [shift={(570,210)}, rotate = 270] [color={rgb, 255:red, 0; green, 0; blue, 0 }  ][line width=0.75]    (0,5.59) -- (0,-5.59)   ;
\draw [shift={(570,180)}, rotate = 270] [color={rgb, 255:red, 0; green, 0; blue, 0 }  ][line width=0.75]    (0,5.59) -- (0,-5.59)   ;

\draw (630,12) node [anchor=north west][inner sep=0.75pt]   [align=left] {$\displaystyle E_{0}$};
\draw (630,90) node [anchor=north west][inner sep=0.75pt]   [align=left] {$\displaystyle E_{1}$};
\draw (630,202) node [anchor=north west][inner sep=0.75pt]   [align=left] {$\displaystyle E_{2}$};
\draw (572,86.5) node [anchor=north west][inner sep=0.75pt]   [align=left] {$\displaystyle \varepsilon _{1}$};
\draw (572,226.5) node [anchor=north west][inner sep=0.75pt]   [align=left] {$\displaystyle \varepsilon _{1}$};
\draw (572,183) node [anchor=north west][inner sep=0.75pt]   [align=left] {$\displaystyle \varepsilon _{2}$};
\end{tikzpicture}
        \caption{Venetian Blind Construction}
        \label{E_n-construction}
    \end{figure}
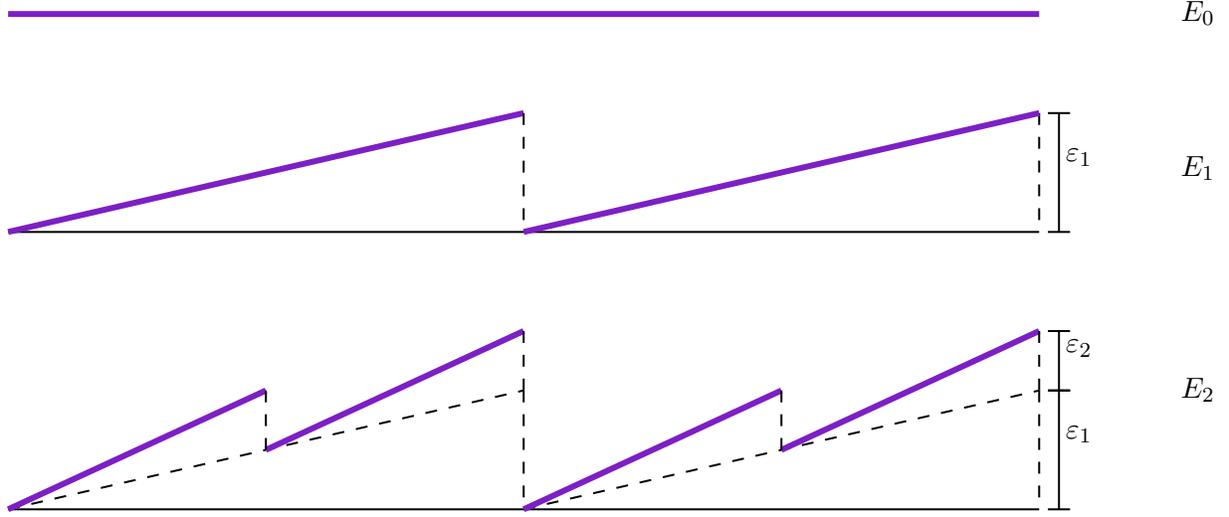

    \medskip
    Our lemma will follow by establishing these four claims: 
    \begin{enumerate}
        \item $E = \lim\limits_{n \to \infty} E_n$ under the Hausdorff metric 
        \item $E$ is totally disconnected.
        \item $\pi_\theta(E) \supset \pi_\theta(E_0)$ for all $\theta \in [0, \frac{\pi}{2}]$. 
        \item $E$ is a rectifiable 1-set. 
    \end{enumerate} 

    \noindent(1). To prove the first claim, notice that $E_{i+1} \subset (E_i)_{\varepsilon_{i+1}}$. So $d_H(E_i, E_{i+1}) < \varepsilon_{i+1}$. Then, $d_H(E_m, E_n) < \sum_{i=n}^m \varepsilon_{i+1}$ by the triangle inequality. If we choose $\{\varepsilon_i\}$ to be a summable sequence, we see that $E_n$ is a Cauchy sequence in the Hausdorff metric. By completeness, $E_n$ converges to some compact set $E$.

\medskip

   \noindent (2). We will now show the second claim. 
    The vertical segment from $(\frac{1}{2}, 0)$ to $(\frac{1}{2}, t)$ where $t \in [0, \varepsilon_1]$ is not in $E$. So there exist bounded open sets $U_0$ and $U_1$ such that 
    \[E \cap \left\{(x, y) : 0 \leq x < \frac{1}{2} \right \} \subseteq U_0 \; \; \text{ and } \; \; E \cap \left\{(x, y) : \frac{1}{2} \leq x < 1 \right \} \subseteq U_1.\]
    For induction, suppose $U_\sigma$, $\sigma\in \{0,1\}^n$, is an open set such that 
    \[E \cap \left \{ (x, y) : \frac{i}{2^n} \leq x < \frac{i+1}{2^n} \right \} \subseteq U_\sigma\]
    By the Venetian blind construction, there exist vertical jumps whenever the $x$-coordinate is a dyadic point, giving us that
    \[E \cap \left \{ (x, y) : \frac{2i}{2^{n+1}} \leq x < \frac{2i+1}{2^{n+1}} \right \} \subseteq U_{\sigma_0} \; \; \text{ and } \; \; E \cap \left \{ (x, y) : \frac{2i+1}{2^{n+1}} \leq x < \frac{2i+2}{2^{n+1}} \right \} \subseteq U_{\sigma_1}, \]
    where $U_{\sigma_0}$ and $U_{\sigma_1}$ are both open sets inside $U_{\sigma}$. Thus, 
    \[E \subset \bigcap_{n=1}^\infty \bigcup_{\sigma \in \Sigma^n} U_\sigma, \; \; \text{ where } \Sigma = \{0, 1\}.\]
    All $U_\sigma$ are disjoint from one another by construction, and one can ensure the maximum of the diameters of $U_{\sigma}$ among $\sigma\in \Sigma^{n}$ tends to zero as $n\to\infty$ since $\{\epsilon_n\}$ is summable. Hence $E$ is totally disconnected.
    %

\medskip

    \noindent{(3).} The third claim follows from showing that any line with slope $-\tan \theta$ for $\theta \in [0, \frac{\pi}{2}]$ passing through $E_0$ must also pass through $E$. Let $L$ be such a line through $E_0$. By the Venetian blind construction, $L$ must also pass through $E_n$ for all $n$, i.e. $L\cap E_n\ne\varnothing$. 
    As $E_n \to E$ in the Hausdorff metric, we see $L$ must also pass through $E$. Hence $\pi_\theta(E) \supset \pi_\theta(E_0)$.

\medskip

    \noindent(4) For the fourth claim, we will show that $E$ is indeed rectifiable by constructing a curve with finite ${\mathcal H}^1$ measure containing $E$. Define the curve $\Gamma_1$ as 
    \[\Gamma_1 = E_1 \cup {\mathcal V}_1 \]
    where  $${\mathcal V}_1 = L\left( \left( \frac{1}{2}, 0 \right), \; \left( \frac{1}{2}, \varepsilon_1 \right)  \right) \cup L\left( (1, 0) \; (1, \varepsilon_1)\right): = v_{1,1}\cup v_{2,1}.
    $$
    That is, $\Gamma_1$ is the union of $E_1$ with the collection $\mathcal{V}_1$ of vertical line segments that connect the disjoint line segments of $E_1$. 
        Suppose $\Gamma_j$ is constructed, such that
    \[\Gamma_j = E_j \cup \mathcal{V}_j.\]
   We construct $\Gamma_{j+1}$ as follows. Recall that $E_j = \bigcup_{k=1}^{2^j} L_{k, j}$. For each $r = 1,2,...,2^{j+1}$, we consider $L_{k_r, j}$ to be the unique line segment with the largest $y$-coordinate that intersects $x = \frac{r}{2^{j+1}}$. Define $v'_{r, j+1}$ to be the vertical line segment at $x = \frac{r}{2^{j+1}}$ of length $\varepsilon_{j+1}$,  beginning at $L_{k_r, j}$, and moving vertically $\varepsilon_{j+1}$ units. Then we have $\mathcal{V}_{j+1}$ is the collection of vertical line segments $v_{r, j+1}$ defined as 
    $$
        v_{r, j+1} = \begin{cases}
            v'_{r, j+1} \; & \; \text{ if $r$ is odd} \\
            v_{r/2, j} \cup v'_{r, j+1} \; & \; \text{ if $r$ is even.}
        \end{cases}
    $$
    Then we can define $\Gamma_{j+1}$ as 
    \[\Gamma_{j+1} = E_{j+1} \cup \mathcal{V}_{j+1}.\]
    $\Gamma_3$ is represented by the bolded line in Figure \ref{Gamma-3}. 

    \begin{figure}
        \centering

\tikzset{every picture/.style={line width=0.75pt}} 

\begin{tikzpicture}[x=0.75pt,y=0.75pt,yscale=-1,xscale=1]

\draw [color={rgb, 255:red, 0; green, 0; blue, 0 }  ,draw opacity=1 ][line width=0.75]    (50,224) -- (570,224) ;
\draw [color={rgb, 255:red, 0; green, 0; blue, 0 }  ,draw opacity=1 ][line width=0.75]  [dash pattern={on 4.5pt off 4.5pt}]  (50,224) -- (310,164) ;
\draw [color={rgb, 255:red, 0; green, 0; blue, 0 }  ,draw opacity=1 ][line width=0.75]  [dash pattern={on 4.5pt off 4.5pt}]  (310,224) -- (570,164) ;
\draw [color={rgb, 255:red, 0; green, 0; blue, 0 }  ,draw opacity=1 ][line width=0.75]  [dash pattern={on 4.5pt off 4.5pt}]  (50,224) -- (180,164) ;
\draw [color={rgb, 255:red, 0; green, 0; blue, 0 }  ,draw opacity=1 ][line width=0.75]  [dash pattern={on 4.5pt off 4.5pt}]  (180,194) -- (310,134) ;
\draw [color={rgb, 255:red, 0; green, 0; blue, 0 }  ,draw opacity=1 ][line width=0.75]  [dash pattern={on 4.5pt off 4.5pt}]  (440,194) -- (570,134) ;
\draw [color={rgb, 255:red, 0; green, 0; blue, 0 }  ,draw opacity=1 ][line width=0.75]  [dash pattern={on 4.5pt off 4.5pt}]  (310,224) -- (440,164) ;
\draw    (580,164) -- (580,191) -- (580,224) ;
\draw [shift={(580,224)}, rotate = 270] [color={rgb, 255:red, 0; green, 0; blue, 0 }  ][line width=0.75]    (0,5.59) -- (0,-5.59)   ;
\draw [shift={(580,164)}, rotate = 270] [color={rgb, 255:red, 0; green, 0; blue, 0 }  ][line width=0.75]    (0,5.59) -- (0,-5.59)   ;
\draw    (580,134) -- (580,161) -- (580,164) ;
\draw [shift={(580,164)}, rotate = 270] [color={rgb, 255:red, 0; green, 0; blue, 0 }  ][line width=0.75]    (0,5.59) -- (0,-5.59)   ;
\draw [shift={(580,134)}, rotate = 270] [color={rgb, 255:red, 0; green, 0; blue, 0 }  ][line width=0.75]    (0,5.59) -- (0,-5.59)   ;
\draw [color={rgb, 255:red, 124; green, 31; blue, 197 }  ,draw opacity=1 ][line width=2.25]    (50,224) -- (115,179) ;
\draw [color={rgb, 255:red, 124; green, 31; blue, 197 }  ,draw opacity=1 ][line width=2.25]    (115,194) -- (180,149) ;
\draw [color={rgb, 255:red, 124; green, 31; blue, 197 }  ,draw opacity=1 ][line width=2.25]    (505,164) -- (570,119) ;
\draw [color={rgb, 255:red, 124; green, 31; blue, 197 }  ,draw opacity=1 ][line width=2.25]    (440,194) -- (505,149) ;
\draw [color={rgb, 255:red, 124; green, 31; blue, 197 }  ,draw opacity=1 ][line width=2.25]    (375,194) -- (440,149) ;
\draw [color={rgb, 255:red, 124; green, 31; blue, 197 }  ,draw opacity=1 ][line width=2.25]    (310,224) -- (375,179) ;
\draw [color={rgb, 255:red, 124; green, 31; blue, 197 }  ,draw opacity=1 ][line width=2.25]    (245,164) -- (310,119) ;
\draw [color={rgb, 255:red, 124; green, 31; blue, 197 }  ,draw opacity=1 ][line width=2.25]    (180,194) -- (245,149) ;
\draw    (580,120) -- (580,131) -- (580,134) ;
\draw [shift={(580,134)}, rotate = 270] [color={rgb, 255:red, 0; green, 0; blue, 0 }  ][line width=0.75]    (0,5.59) -- (0,-5.59)   ;
\draw [shift={(580,120)}, rotate = 270] [color={rgb, 255:red, 0; green, 0; blue, 0 }  ][line width=0.75]    (0,5.59) -- (0,-5.59)   ;
\draw [color={rgb, 255:red, 124; green, 31; blue, 197 }  ,draw opacity=1 ][line width=2.25]    (115,194) -- (115,179) ;
\draw [color={rgb, 255:red, 124; green, 31; blue, 197 }  ,draw opacity=1 ][line width=2.25]    (375,194) -- (375,179) ;
\draw [color={rgb, 255:red, 124; green, 31; blue, 197 }  ,draw opacity=1 ][line width=2.25]    (310,224) -- (310,119) ;
\draw [color={rgb, 255:red, 124; green, 31; blue, 197 }  ,draw opacity=1 ][line width=2.25]    (245,164) -- (245,149) ;
\draw [color={rgb, 255:red, 124; green, 31; blue, 197 }  ,draw opacity=1 ][line width=2.25]    (180,194) -- (180,149) ;
\draw [color={rgb, 255:red, 124; green, 31; blue, 197 }  ,draw opacity=1 ][line width=2.25]    (570,224) -- (570,119) ;
\draw [color={rgb, 255:red, 124; green, 31; blue, 197 }  ,draw opacity=1 ][line width=2.25]    (505,164) -- (505,149) ;
\draw [color={rgb, 255:red, 124; green, 31; blue, 197 }  ,draw opacity=1 ][line width=2.25]    (440,194) -- (440,149) ;

\draw (591,180) node [anchor=north west][inner sep=0.75pt]   [align=left] {$\displaystyle \varepsilon _{1}$};
\draw (591,140) node [anchor=north west][inner sep=0.75pt]   [align=left] {$\displaystyle \varepsilon _{2}$};
\draw (591,112) node [anchor=north west][inner sep=0.75pt]   [align=left] {$\displaystyle \varepsilon _{3}$};

\end{tikzpicture}
        \caption{$\Gamma_3$ curve}
        \label{Gamma-3}
    \end{figure}
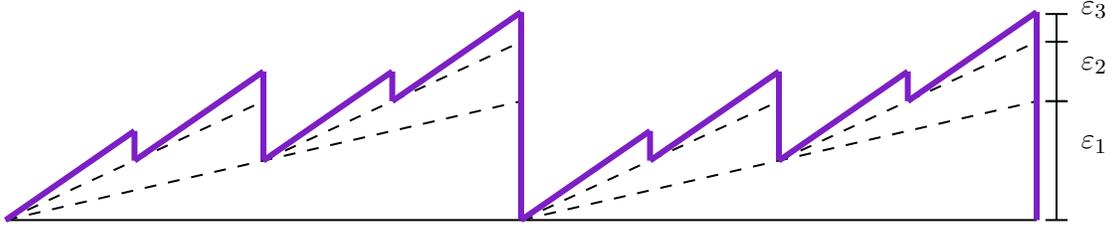

    In the proof of the first claim, we showed that $E_n \to E$ under $d_H$. Since the maximal distance between $\Gamma_n$ and $\Gamma_{n+1}$ is $\varepsilon_{n+1}$ under $d_H$, we can apply the same strategy to show that $\Gamma_n$ converges to some compact set $\Gamma$. We now claim  that $E \subset \Gamma$:

    \medskip
    Fix $\varepsilon > 0$. For every $x \in E$, there exists some sequence of points $\{x_n\}_{n \in \N}$ such that each $x_n \in E_n$ and $|x_n - x| \to 0$. That is, there exists some $N_1 \in \N$ such that $n \geq N_1 $ implies that $|x_n - x| < \frac{\varepsilon}{2}$. As $E_n\subset \Gamma_n$, $x_n\in \Gamma_n$.     Because $E_n \subset \Gamma_n$, we have $x_n \in \Gamma_n$.    By the convergence of $\Gamma_n \to \Gamma$ as $n \to \infty$ in $d_H$, we have that there exists some $N_2 \in \N$ such that $n \geq N_2$ implies  that $d_H(\Gamma_n,\Gamma) < \frac{\varepsilon}{2}$. Choose $N = \max(N_1, N_2)$. Then for all $n \geq N$, we have $x_n\in (\Gamma)_{\varepsilon/2}$ and hence $x\in (\Gamma)_{\varepsilon}$ where we recall that $(\Gamma)_{\varepsilon}$ is the $\varepsilon$-neighborhood of $\Gamma$.     Because $\varepsilon$ is arbitrary and $\Gamma$ is compact, we take $\varepsilon \to 0$ and see that $x \in \Gamma$. Hence $E \subset \Gamma$.

\medskip

    To complete our proof, we will need the following theorem from \cite[p.39]{falconer1985geometry}. Recall that a continuum is a compact connected metric space. 
    \begin{theorem} \label{Falconer85}
        Let $\{\Gamma_n\}$ be a sequence of continua in $\R^d$ convergent in the Hausdorff metric to a compact set $\Gamma$. Then $\Gamma$ is a continuum and 
        \[\mathcal{H}^1(\Gamma) \leq \liminf_{n \to \infty} \mathcal{H}^1(\Gamma_n).\]
    \end{theorem}
   $\Gamma_n$ is compact and connected.  Therefore $\{\Gamma_n\}$ is indeed a sequence of continua that converges to $\Gamma$ under $d_H$. We can then apply Theorem \ref{Falconer85} to see that $\Gamma$ itself is a continuum - and hence connected - as well as that 
    \[\mathcal{H}^1(\Gamma) \leq \liminf_{n \to \infty} \mathcal{H}^1(\Gamma_n)\]
    It remains to compute $\mathcal{H}^1 (\Gamma_n)$ which is exactly the total  length of the line segments in $E_n\cup{\mathcal V}_n$.
    Note that 
    $$
    {\mathcal H}^1({\mathcal V}_n) = \sum_{j=1}^{n}2^j\varepsilon_j
    $$
    and by elementary geometry and the Pythagorean theorem, the length of one of the ``blinds" in $E_n$ is equal to 
    $$
    \sqrt{\left(\frac{1}{2^n}\right)^2+ \left(\frac{\varepsilon_1}{2^{n-1}}+\frac{\varepsilon_2}{2^{n-2}}+\cdots+\varepsilon_n\right)^2} = \frac1{2^n} \cdot \sqrt{1+\left( \sum_{j=1}^n 2^j \varepsilon_j \right)^2}.
    $$
    Therefore, summing all $2^n$ ``blinds" which are of the same length, we have
    $$
    {\mathcal H}^1(\Gamma_n) = \sum_{j=1}^n 2^j \varepsilon_j + \sqrt{1+ \left( \sum_{j=1}^n 2^j \varepsilon_j \right)^2}. 
    $$
    If we choose our rapidly decreasing positive sequence $\{\varepsilon_j\}_{j=1}^\infty$ such that $\sum_{j=1}^\infty 2^j \varepsilon_j  < \infty,$ then we see that 
    \begin{align*}
        {\mathcal H^1}(\Gamma)\le \liminf_{n \to \infty} {\mathcal H}^1(\Gamma_n) <\infty.
    \end{align*}
    Additionally, since $\Gamma$ contains vertical line segments at each dyadic point, we know that $\mathcal{H}^1(\Gamma) > 0$. Therefore $\Gamma$ is indeed a 1-set.  

    \medskip
    

    Finally, by the monotonicity of measure,  $\mathcal{H}^1(E) \leq \mathcal{H}^1(\Gamma) < \infty$. For a contradiction, suppose that $\mathcal{H}^1(E) = 0$. By the Bescovitch projection theorem, $E$ will be purely unrectificable and almost all projection has ${\mathcal H}^1$ measure zero.     However by the third claim, we know that $1 \geq \mathcal{H}^1(\pi_\theta(E))$ for all $\theta \in [0, \frac{\pi}{2}]$. To avoid a contradiction, we deduce that $0 < \mathcal{H}^1(E) < \infty$. Thus $E$ is a 1-set and rectifiable, justifying our last claim.
    \end{proof}

We will use Lemma \ref{tot-discon-rect-set-lemma} to prove Theorem \ref{theorem_Main1} (2). The proof in high dimensions requires an induction, so we first settle it for a polygon on $\R^2$. 

\begin{proposition}\label{prop_dim2} Let $C$ be a convex polygon on $\R^2$. Then there exists a compact totally disconnected rectifiable 1-set $K$ such that $\pi_\theta(K) = \pi_\theta([0, 1]^2)$ for all $\theta \in [0, \pi)$. Hence, polygons are fractal decomposable by a compact totally disconnected rectificable 1-set.     
\end{proposition}
  
\begin{proof} Let $C$ be a $n$-sided polygon, and let $L_1, L_2, L_3,...,L_n$  be the sides of the polygon. We claim that for each $j = 1,...,n$, there exists a compact totally disconnected rectifiable set $E_j\subset C$ such that $\pi_{\theta}(E_j) \supset \pi_{\theta}(L_j)$ for all $\theta\in[0,\pi)$. Then we take $E$ to be the union of $E_j$, which is still a compact totally disconnected rectifiable set. The union is a subset of $C$, and since $\pi_{\theta}(C) = \bigcup_{j=1}^n\pi_{\theta}(L_j)$, this union gives  us 
$$
\pi_{\theta}(C) \subset \bigcup_{j=1}^n\pi_{\theta}(E_j) = \pi_{\theta}(E). 
$$
But $E\subset C,$ and we have the other inclusion for the projection as well. Thus $E$ is our desired set. 

\medskip

To justify the claim  we take a side $L = L_j$ for some $j$. Suppose that $L$  connects with the other two sides, called $L'$ and $L''$, at some vertices. We now cut $L = \ell'\cup\ell''$ into two lines where $L'$ connects with $\ell'$ and similarly for the others. Consider $L'\cup \ell'$ and $L''\cup \ell''$.  We have two cases depending on the angle between $L'$ and $\ell'$ (and respectively $L''$ and $\ell''$).

\medskip

(i) Suppose that $L'$ makes an obtuse angle with $\ell'$ (i.e. the angle lies between $[\pi/2, \pi)$). We simply take $E$ in Lemma \ref{tot-discon-rect-set-lemma} for $\ell'$. For $\epsilon_j$ sufficiently small, $E$ must be inside the polygon and  $\pi_{\theta}(E) = \pi_{\theta}(\ell')$ for all $\theta\in [0,\pi)$. 

\medskip

(ii) Suppose that $L'$ makes an acute angle with $\ell'$. Without loss of generality, assume that $\ell'  = [0,1]\times \{0\}$, and $L'$ starts at $(1,0)$, making an acute angle with $\ell'$. Due to the acute angle, the $E$ constructed in Lemma \ref{tot-discon-rect-set-lemma} cannot lie in $C$.  However, we can further decompose $\ell'$ into countable interior disjoint line segments towards the vertex: 
$$
\ell' = \bigcup_{j=1}^{\infty}\ell_j, \ \ell_j = [1-2^{-j},1-2^{-j-1}]\times \{0\}. 
$$
We apply Lemma \ref{tot-discon-rect-set-lemma} for each $\ell_j$ to construct $E_j\subset C$ (by taking sufficiently small $\varepsilon_j$). Define $E = \bigcup_{j=1}^{\infty}{E_j}\cup\{(1,0)\}$. Then $\pi_{\theta}(E)\supset \pi_{\theta}(\ell')$ for all $\theta\in[0,\pi)$. It remains to check if $E$ is closed and totally disconnected.  $E$ is closed because if $x = (x_1,x_2)\in \R^2\setminus E$, then for $x_1>1$ and for sufficiently small $\delta>0$, the ball $B(x,\delta)\subset (1-2^{-j},1-2^{-j-1})\times \R$. But in this strip, $E\cap [1-2^{-j},1-2^{-j-1}]\times \R = E_j$ and $E_j$ is closed. The ball $B(x,\delta)\subset \R^2\setminus E$. Hence, $\R^2\setminus E$ is an open set, showing that $E$ is closed. Finally, it is a routine check that $E$ is totally disconnected using Theorem \ref{theorem_tD} (2), so we will omit this detail. THe proof of this proposition is now complete. 
\end{proof}
\medskip


{\bf Proof of Theorem \ref{theorem_Main1} (2).} We prove the theorem by induction on dimension $d$. By Proposition \ref{prop_dim2}, the theorem has been proven for $d=2$. Suppose that for all $(d-1)$-dimensional convex polytopes $C$, there exists a rectifiable 1-set $K$ such that $K\subset C$ and $K$ projects very thick shadows in every one-dimensional subspace. We now prove that it is also true for dimension $d$. 

\medskip

To prove the statement, let $C$ be a convex polytope of dimension $d$. Suppose first that  $C$ is lying on a $(d-1)$-dimensional hyperplane $H_0$. By the induction hypothesis, we construct the rectifiable 1-set $K$ inside $C\cap H_0$. In this case, for any $(d-1)$-dimensional hyperplanes $H$ such that $H$ intersects $C$, $H\cap H_0$ is a $(d-2)$-dimensional hyperplane which intersects $C$. As $K$ projects very thick shadows in every 1-dimensional subspace, every $d-2$-dimensional subspace must intersect $K$. So $H\cap H_0 \cap K\ne \emptyset$, and hence $H\cap K\ne \emptyset$. 

\medskip

Suppose now that $C$ is not lying in any $(d-1)$-dimensional hyperplane. Then $C$ admits a half-space representation: 
$$
C = \bigcap_{i=1}^M {\mathcal H}_i,
$$
where ${\mathcal H}_i$ are closed half-spaces. Moreover, each ${\mathcal H}_i\cap C$ is a $(d-1)$-dimensional convex polytope. We now construct rectifiable 1-sets $K_i$ for each ${\mathcal H}_i\cap C$ by the induction hypothesis, and define 
$$
K = \bigcup_{i=1}^M K_i.
$$  
Then each $(d-1)$-dimensional hyperplane must intersect one of the faces ${\mathcal H}_i\cap C$. By the induction hypothesis, the hyperplane $H_i$ must intersect $K_i$. Thus, every $(d-1)$-dimensional hyperplane must intersect $K$. As $K$ is a finite union of compact totally disconnected rectifiable $1$-sets, it must be a compact totally disconnected rectifiable $1$-set as well. The proof is complete. 
 




\medskip

\section{Remarks and open questions}

This paper provides a detailed study regarding the projections of very thick shadows onto 1-dimensional subspaces for polytopes on $\R^d$. We conclude this paper with this section by discussing more general cases.

\subsection{Other Convex sets.} First, we establish the following proposition showing that 1-dimensional fractal decomposability is not possible for general convex sets. The definition of exposed points can be found in Section 2. 

\medskip

\begin{proposition}\label{theorem_exposed}
    Let $C$ be a closed convex set in $\R^d$. 
Suppose that $C$ is fractal decomposable with the totally disconnected Borel set $K$. Then all exposed points of $C$ are also in $K$, and the set of exposed points is totally disconnected.
\end{proposition}

\begin{proof}
    Let $x$ be an exposed point of $C$. There exists some hyperplane $H$ supporting $x$ such that $H \cap (C \setminus \{x\}) = \varnothing$. We now project $C$ and $K$ onto the orthogonal complement of $H$, denoted by $H^\perp$, which is a $1$-dimensional subspace. Since $C$ is fractal decomposable, we have that 
    $\pi_{H^\perp}(C) = \pi_{H^\perp}(K). $
    Therefore, there must exist some $y \in K$ such that 
    \[\pi_{H^\perp}(x) = \pi_{H^\perp}(y).\]
    Note that if $x \in H$, then 
    \[\pi_{H^\perp}^{-1}(\pi_{H^\perp}(x)) = H.\]
    We thus know that $y \in \pi_{H^\perp}^{-1}(\pi_H^\perp(x)) = H$. At the same time, $y \in K \subset C$. So $y \in C \cap H = \{x\}$. Hence, $y = x$, and we proved that the exposed points of $C$ are also in $K$. Because $K$ is totally disconnected and the set of all exposed points of $C$ is a subset of $K$, we have that the set of exposed points is also totally disconnected. 
\end{proof}

 Consequently, it is impossible for convex sets with a smooth boundary and everywhere positive Gaussian curvature, like the Euclidean ball, to be 1-dimensional fractal decomposable using totally disconnected sets. Fractal decomposability remains an open question for those convex sets with totally disconnected exposed points, but that are not polytopes.

\subsection{Higher dimensional fractal decomposability.} If we now try to project onto two or higher dimensional subspaces, total disconnectedness can no longer be obtained even for convex polytopes. The following proposition works on $\R^3$, and it clearly works on any greater dimensions.

\begin{proposition}
    Suppose that a closed convex polytope $C$ on $\R^3$ is 2-dimensional fractal decomposable with a compact set $K$. Then all edges of the polytope must be in the set $K$. Consequently, $K$ cannot be totally disconnected. 
\end{proposition}

\begin{proof}
    Note that for each edge $E$, we can always find a supporting hyperplane $H$ such that $H\cap C = E$. If $C$ is 2-dimensional fractal decomposable with compact set $K$, then every line passing through $C$ must also pass through $K$. However, for each $x\in E$,  we can take a line $\ell$ in $H$ that is orthogonal to $E$ such that $\{x\} = \ell\cap E$. Then  $\ell\cap K = \{x\}$, meaning that $E\subset K$. This completes the proof. 
\end{proof}

As a result, we do not have 2-dimensional fractal decomposability with a totally disconnected compact set $K$ for a convex polytope. We need to replace total disconnectedness with some other type of connectivity. On the other hand, the unit cube $[0,1]^3$ admits a trivial solution for the 2-dimensional fractal decomposability. One can subdivide the unit cube into $n^3$ smaller cubes of sidelength $1/n$ and choose cubes that intersect the boundary of $[0,1]^3$ . The self-similar fractal will contain the boundary of $[0,1]^3$ which is trivially 2-dimensional fractal decomposable. A simple question to ask here which avoids a trivial answer is:

\medskip

{\bf (Qu):} Is it possible to construct a self-similar fractal $K$ on $[0,1]^3$ such that $[0,1]^3$ is 2-dimensional fractal decomposable with $K$ and 
$$
\dim_H(P\cap K)<2
$$
for all affine hyperplanes $P$ that passes through $[0,1]^3$? 

\medskip

In some sense, we expect $K$ should be ``plane-free". Another question worth considering is that if $K$ is totally disconnected, how many hyperplanes can we choose to guarantee that the projection of $K$ is equal to that of the convex hull on those hyperplanes?

\medskip

Finally, our Theorem \ref{theorem_Main1}(1) showed that every polytope is 1-dimensional fractal decomposable using a finite union of self-affine sets. We also showed that  triangles also admit  self-similar solutions. It will be interesting to see if every convex polytope is 1-dimensional fractal decomposable using only one totally disconnected self-similar or self-affine sets. It is also unclear if higher dimensional simplices admit a self-similar solution as simplices themselves are no longer self-similar on higher dimensions (see e.g. \cite{Hertel}).

\medskip

\noindent{\bf Acknowledgement.} The authors would like to thank Alan Chang, Yeonwook Jung, Ruxi Shi, Pablo Shmerkin and Jun Jason Luo for many helpful discussions about different parts of the paper. In particular, they  would like to thank Pablo Shmerkin for guiding us to relevant background papers as well as Alan Chang and Tuomas Oprenon for sharing with us a workable idea for the proof in Section 6.  The research of Lekha Priya Patil was supported the ARCS Foundation and the BAMM! scholarship funded by the National Science Foundation during her Master's study. LPP acknowledges their generous support.

\printbibliography

\end{document}